\documentclass[11pt]{article}
\usepackage[hypertexnames=false]{hyperref}
\usepackage{amsfonts, amsmath, amssymb, amsthm, mathtools, bm, tikz, enumitem, etoolbox}

\usepackage[capitalize, noabbrev]{cleveref}
\crefname{equation}{}{}

\usepackage{fullpage}
\usepackage{microtype}
\usetikzlibrary{decorations.pathreplacing}
\tikzset{
  vertex/.style = {circle, draw, fill=white, inner sep=0pt, minimum width=4pt},
  root/.style = {circle, draw, fill=black, inner sep=0pt, minimum width=4pt},
  bracket/.style = {decorate,decoration={brace,amplitude=5pt},xshift=0pt,yshift=-10pt},
  bracket-r/.style = {decorate,decoration={brace,amplitude=5pt},xshift=10pt,yshift=0pt},
  bracket-u/.style = {decorate,decoration={brace,amplitude=5pt},xshift=0pt,yshift=10pt},
}

\title{Negligible obstructions and Tur\'an exponents}
\author{
  Tao Jiang\thanks{Department of Mathematics, Miami Univeristy, Oxford, OH 45056, USA. Email: {\tt jiangt@miamioh.edu}. Supported in part by U.S. taxpayers through the National Science Foundation (NSF) grant DMS-1855542.}
  \and Zilin Jiang\thanks{School of Mathematical and Statistical Sciences, and School of Computing and Augmented Intelligence, Arizona State University, Tempe, AZ 85281. Email: {\tt zilinj@asu.edu}. The work was done when Z.~Jiang was an Applied Mathematics Instructor at Massachusetts Institute of Technology, and was supported in part by an AMS Simons Travel Grant, and by U.S. taxpayers through NSF grant DMS-1953946.}
  \and Jie Ma\thanks{School of Mathematical Sciences, University of Science and Technology of China, Hefei 230026, P.R.~China. Email: {\tt jiema@ustc.edu.cn}. Supported in part by the National Key R\&D Program of China 2020YFA0713100, National Natural Science Foundation of China grants 11622110 and 12125106, and Anhui Initiative in Quantum Information Technologies grant AHY150200.}}
\date{}

\newtheorem{theorem}{Theorem}
\newtheorem{lemma}[theorem]{Lemma}
\newtheorem{proposition}[theorem]{Proposition}
\newtheorem{corollary}[theorem]{Corollary}
\newtheorem{conjecture}[theorem]{Conjecture}

\theoremstyle{definition}
\newtheorem{definition}[theorem]{Definition}

\theoremstyle{remark}
\newtheorem*{remark}{Remark}

\newtheorem{claim}{Claim}
\newtheorem*{claim*}{Claim}
\Crefname{claim}{Claim}{Claims}

\newenvironment{claimproof}[1][Proof of Claim]{\begin{proof}[#1]}{\end{proof}}

\DeclarePairedDelimiter\abs{\lvert}{\rvert}%
\DeclarePairedDelimiter\set{\{}{\}}%
\DeclarePairedDelimiter\floor{\lfloor}{\rfloor}%

\DeclareMathOperator{\ex}{ex}
\DeclareMathOperator{\inj}{inj}
\DeclareMathOperator{\Inj}{Inj}
\DeclareMathOperator{\amp}{amp}
\DeclareMathOperator{\Amp}{Amp}
\DeclareMathOperator{\ext}{ext}
\DeclareMathOperator{\Ext}{Ext}
\DeclareMathOperator{\EE}{E}

\newcommand{\F}{\mathcal{F}}

\newcommand{\U}{\mathcal{U}}
\newcommand{\N}{\mathbb{N}}

\newcommand{\eps}{\varepsilon}

\newcommand{\dset}[2]{\set{#1 \colon #2}}

\newcommand{\absng}{\abs{N_G(U_0)}}
\newcommand{\babsng}{\binom{\absng}{s+1}}
\newcommand{\us}{U_\sigma}

\newcommand{\hs}{H_\sigma}
\newcommand{\wist}{\widetilde{I}_\sigma^\times}
\newcommand{\wths}{\widetilde{H}_\sigma}
\newcommand{\wtus}{\widetilde{U}_\sigma}
\newcommand{\wtws}{\widetilde{W}_\sigma}
\newcommand{\wtas}{\widetilde{A}_\sigma}

\begin{document}

\maketitle

\begin{abstract}
  We show that for every rational number $r \in (1,2)$ of the form $2 - a/b$, where $a, b \in \mathbb{N}^+$ satisfy $\lfloor b/a \rfloor^3 \le a \le b / (\lfloor b/a \rfloor +1) + 1$, there exists a graph $F_r$ such that the Tur\'an number $\operatorname{ex}(n, F_r) = \Theta(n^r)$. Our result in particular generates infinitely many new Tur\'an exponents. As a byproduct, we formulate a framework that is taking shape in recent work on the Bukh--Conlon conjecture.
\end{abstract}



\section{Introduction} \label{sec:intro}

Given a family $\F$ of graphs, the Tur\'an number $\ex(n, \F)$ is defined to be the maximum number of edges in a graph on $n$ vertices that contains no graph from the family $\F$ as a subgraph. The classical Erd\H{o}s--Stone--Simonovits theorem shows that arguably the most interesting problems about Tur\'an numbers, known as the degenerate extremal graph problems, are to determine the order of magnitude of $\ex(n, \F)$ when $\F$ contains a bipartite graph. The following conjecture attributed to Erd\H{o}s and Simonovits is central to Degenerate Extremal Graph Theory (see \cite[Conjecture~1.6]{FS13}).

\begin{conjecture}[Rational Exponents Conjecture]
  For every finite family $\F$ of graphs, if $\F$ contains a bipartite graph, then there exists a rational $r \in [1,2)$ and a positive constant $c$ such that $\ex(n, \F) = cn^r + o(n^r)$.
\end{conjecture}

Recently Bukh and Conlon made a breakthrough on the inverse problem \cite[Conjecture 2.37]{FS13}.

\begin{theorem}[Bukh and Conlon~\cite{BC18}] \label{thm:bc-main}
  For every rational number $r\in (1,2)$, there exists a finite family of graphs $\F_r$ such that $\ex(n, \F_r) = \Theta(n^r)$.
\end{theorem}

Motivated by another outstanding problem of Erd\H{o}s and Simonovits (see \cite[Section III]{E81} and \cite[Problem 8]{E88}), subsequent work has been focused on the following conjecture, which aims to narrow the family $\F_r$ in \cref{thm:bc-main} down to a single graph.

\begin{conjecture}[Realizability of Rational Exponents] \label{conj:main}
  For every rational number $r\in (1,2)$, there exists a bipartite graph $F_r$ such that $\ex(n, F_r) = \Theta(n^r)$.\footnote{Erd\H{o}s and Simonovits asked a much stronger question: for every rational number $r\in (1,2)$, find a bipartite graph $F_r$ such that $\ex(n, F_r) = cn^r + o(n^r)$ for some positive constant $c$.}
\end{conjecture}

It is believed that the graph $F_r$ in \cref{conj:main} could be taken from a specific yet rich family of graphs, for which we give the following definitions.

\begin{definition}
  A \emph{rooted graph} is a graph $F$ equipped with a subset $R(F)$ of vertices, which we refer to as \emph{roots}. We define the $p$th \emph{power} of $F$, denoted $F^p$, by taking the disjoint union of $p$ copies of $F$, and then identifying each root in $R(F)$, reducing multiple edges (if any) between the roots.
\end{definition}

\begin{definition} \label{def:rooted}
  Given a rooted graph $F$, we define the \emph{density} $\rho_F$ of $F$ to be $e(F)/(v(F)-\abs{R(F)})$, where $v(F)$ and $e(F)$ denote the number of vertices and respectively edges of $F$. We say that a rooted graph $F$ is \emph{balanced} if $\rho_F > 1$, and for every subset $S$ of $V(F) \setminus R(F)$, the number of edges in $F$ with at least one endpoint in $S$ is at least $\rho_F\abs{S}$ .
\end{definition}

Indeed the next result on Tur\'an numbers, which follows immediately from \cite[Lemma~1.2]{BC18}, establishes the lower bound in \cref{conj:main} for some power of a balanced rooted tree.\footnote{A rooted tree is a rooted graph that is also a tree, not to be confused with a tree having a designated vertex.}

\begin{lemma}[Bukh and Conlon~\cite{BC18}]\label{lem:bc-lb}
  For every balanced rooted tree $F$, there exists $p \in \N^+$ such that $\ex(n, F^p) = \Omega(n^{2-1/\rho_F})$.
\end{lemma}

It is conjectured in \cite{BC18} that the lower bound in \cref{lem:bc-lb} can be matched up to a constant factor.

\begin{conjecture}[The Bukh--Conlon Conjecture] \label{conj:bc}
  For every balanced rooted tree $F$ and every $p \in \N^+$, $\ex(n, F^p) = O(n^{2-1/\rho_F})$.
\end{conjecture}

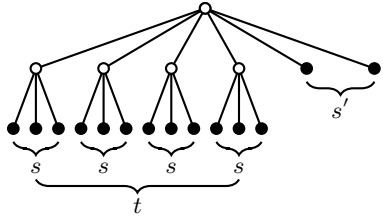
\begin{figure}
  \centering
  \begin{tikzpicture}[thick, scale=0.5]
    \draw (-4.5,-.4) -- (-5.1,-2) node[root]{};
    \draw (-4.5,-.4) -- (-4.5,-2) node[root]{};
    \draw (-4.5,-.4) -- (-3.9,-2) node[root]{};
    \draw [bracket] (-3.9,-2) -- (-5.1,-2) node[midway,yshift=-10pt]{\footnotesize $s$};
    \draw (0,1.2) -- (-4.5,-.4) node[vertex]{};

    \draw (-2.7,-.4) -- (-3.3,-2) node[root]{};
    \draw (-2.7,-.4) -- (-2.7,-2) node[root]{};
    \draw (-2.7,-.4) -- (-2.1,-2) node[root]{};
    \draw [bracket] (-2.1,-2) -- (-3.3,-2) node[midway,yshift=-10pt]{\footnotesize $s$};
    \draw (0,1.2) -- (-2.7,-.4) node[vertex]{};

    \draw (-.9,-.4) -- (-1.5,-2) node[root]{};
    \draw (-.9,-.4) -- (-.9,-2) node[root]{};
    \draw (-.9,-.4) -- (-.3,-2) node[root]{};
    \draw [bracket] (-.3,-2) -- (-1.5,-2) node[midway,yshift=-10pt]{\footnotesize $s$};
    \draw (0,1.2) -- (-.9,-.4) node[vertex]{};

    \draw (.9,-.4) -- (.3,-2) node[root]{};
    \draw (.9,-.4) -- (.9,-2) node[root]{};
    \draw (.9,-.4) -- (1.5,-2) node[root]{};
    \draw [bracket] (1.5,-2) -- (.3,-2) node[midway,yshift=-10pt]{\footnotesize $s$};
    \draw (0,1.2) -- (.9,-.4) node[vertex]{};

    \draw (0,1.2) node[vertex]{} -- (4.5,-.4) node[root]{};
    \draw (0,1.2) node[vertex]{} -- (2.7,-.4) node[root]{};

    \draw [bracket] (4.5,-.4) -- (2.7,-.4) node[midway,yshift=-10pt]{\footnotesize $s'$};

    \draw [bracket] (0.9,-3) -- (-4.5,-3) node[midway,yshift=-10pt]{\footnotesize $t$};
  \end{tikzpicture}
  \caption{$T_{s,t,s'}$ with roots in black.} \label{fig:2-tree}
\end{figure}

Given the fact that every rational number bigger than one indeed appears as the density of some balanced rooted tree (see \cite[Lemma~1.3]{BC18}), \cref{lem:bc-lb,conj:bc} would imply \cref{conj:main}.
Our main result establishes \cref{conj:bc} for certain balanced rooted trees $T_{s,t,s'}$ defined in \cref{fig:2-tree}.

\begin{theorem} \label{thm:main}
  For every $s, t \in \N^+$ and $s' \in \N$, when $s - s' \ge 2$ assume in addition that $t \ge s^3 - 1$. If the rooted tree $F := T_{s,t,s'}$ is balanced, then for every $p \in \N^+$, $\ex(n, F^p) = O(n^{2-1/\rho_F})$, where $\rho_F = (st + t + s')/(t + 1)$.
\end{theorem}

It is not hard to characterize the parameters $s, t, s'$ for which $T_{s,t,s'}$ is balanced.

\begin{proposition} \label{lem:balance-condition}
  For every $s, t \in \N^+$ and $s' \in \N$, the rooted tree $F = T_{s,t,s'}$ is balanced if and only if $\rho_F \ge \max(s,s')$ and $\rho_F > 1$, or equivalently $s' - 1 \le s \le t + s'$ and $(t, s') \neq (1, 0)$. \qed
\end{proposition}

Prior to our work, \cref{conj:bc} has been verified for the balanced rooted trees in \cref{fig:known}: the $K_s^{(0)}$ and $P_t$ cases are classical results due to K\H{o}v\'ari, S\'os and Tur\'an~\cite{KST54}, and respectively Faudree and Simonovits~\cite{FS83}; $Q_{s,1}$ and $S_{2,1,0}$ are due to Jiang, Ma and Yepremyan~\cite{JMY18}; $Q_{s,t}$ and $T_{4,7}$ are due to Kang, Kim and Liu~\cite{KKL18}; $K_s^{(1)}$ and $S_{s,t,0}$ are due to Conlon, Janzer and Lee~\cite{CJL19}; $K^{(2)}_s$ and $K^{(3)}_s$ are due to Jiang and Qiu~\cite{JQ19}; $K^{(t)}_s$ is due to Janzer~\cite{J20}; and $S_{s,t,t'}$ for all $t' \le t$ is very recently settled by Jiang and Qiu~\cite{JQ20}.

\begin{figure}[b]
  \centering
  \begin{tikzpicture}[thick, scale=0.45, baseline=(v.base)]
    \coordinate (v) at (0,0);
    \draw (-3,0) -- (-2,-0.8) node[vertex]{} -- (-1,-0.8) node[vertex]{} -- (0,-0.8) node[vertex]{} -- (1,-0.8) node[vertex]{} -- (2,-0.8) node[root]{};
    \draw (-3,0) -- (-2,0) node[vertex]{} -- (-1,0) node[vertex]{} -- (0,0) node[vertex]{} -- (1,0) node[vertex]{} -- (2,0)node[root]{};
    \draw (-3,0) node[vertex]{} -- (-2,0.8) node[vertex]{} -- (-1,0.8) node[vertex]{} -- (0,0.8) node[vertex]{} -- (1,0.8) node[vertex]{} -- (2,0.8) node[root]{};
    \draw [bracket-r] (2,0.8) -- (2,-0.8) node[black,midway,xshift=10pt]{\footnotesize $s$};
    \draw [bracket] (1,-0.8) -- (-2,-0.8) node[black,midway,yshift=-10pt]{\footnotesize $t$};
    \node [below] at (0,-2.9) {$K^{(t)}_s$};
  \end{tikzpicture}\qquad%
  \begin{tikzpicture}[thick, scale=0.45, baseline=(v.base)]
    \coordinate (v) at (0,0);
    \draw (-2,0) node[root]{} -- (-1,0) node[vertex]{} -- (0,0) node[vertex]{} -- (1,0) node[vertex]{} -- (2,0) node[vertex]{} -- (3,0) node[root]{};
    \draw [bracket] (2,0) -- (-1,0) node[black,midway,yshift=-10pt]{\footnotesize $t$};
    \node [below] at (0.5,-3.2) {$P_t$};
  \end{tikzpicture}\qquad%
  \begin{tikzpicture}[thick, scale=0.45, baseline=(v.base)]
    \coordinate (v) at (0,-0.8);
    \draw (0,1.2) node[root]{} -- (0,0.4);
    \draw (-0.6,1.2) node[root]{} -- (0,0.4);
    \draw (0.6,1.2) node[root]{} -- (0,0.4);
    \draw [bracket-u] (-0.6,1.2) -- (0.6,1.2) node[black,midway,yshift=10pt]{\footnotesize $s$};
    \draw (-2.7,-.4) -- (-3.3,-1.2) node[root]{};
    \draw (-2.7,-.4) -- (-2.7,-1.2) node[root]{};
    \draw (-2.7,-.4) -- (-2.1,-1.2) node[root]{};
    \draw [bracket] (-2.1,-1.2) -- (-3.3,-1.2) node[black,midway,yshift=-10pt]{\footnotesize $s$};
    \draw (0,.4) -- (-2.7,-.4) node[vertex]{};
    \draw (-.9,-.4) -- (-1.5,-1.2) node[root]{};
    \draw (-.9,-.4) -- (-.9,-1.2) node[root]{};
    \draw (-.9,-.4) -- (-.3,-1.2) node[root]{};
    \draw [bracket] (-.3,-1.) -- (-1.5,-1.2) node[black,midway,yshift=-10pt]{\footnotesize $s$};
    \draw (0,.4) -- (-.9,-.4) node[vertex]{};
    \draw (.9,-.4) -- (.3,-1.2) node[root]{};
    \draw (.9,-.4) -- (.9,-1.2) node[root]{};
    \draw (.9,-.4) -- (1.5,-1.2) node[root]{};
    \draw [bracket] (1.5,-1.2) -- (.3,-1.2) node[black,midway,yshift=-10pt]{\footnotesize $s$};
    \draw (0,.4) -- (.9,-.4) node[vertex]{};
    \draw (2.7,-.4) -- (2.1,-1.2) node[root]{};
    \draw (2.7,-.4) -- (2.7,-1.2) node[root]{};
    \draw (2.7,-.4) -- (3.3,-1.2) node[root]{};
    \draw [bracket] (3.3,-1.2) -- (2.1,-1.2) node[black,midway,yshift=-10pt]{\footnotesize $s$};
    \draw (0,.4) node[vertex]{} -- (2.7,-.4) node[vertex]{};
    \draw [bracket] (2.7,-2.4) -- (-2.7,-2.4) node[black,midway,yshift=-10pt]{\footnotesize $t$};
    \draw (0,-4) node[below]{$Q_{s,t}$};
  \end{tikzpicture}\qquad%
  \begin{tikzpicture}[thick, scale=0.45, baseline=(v.base)]
    \coordinate (v) at (0,0);
    \draw (-3,0) -- (-2,-1.2) node[vertex]{} -- (-1,-1.2) node[vertex]{} -- (0,-1.2) node[vertex]{} -- (1,-1.2) node[root]{};
    \draw (-3,0) -- (-2,-0.4) node[vertex]{} -- (-1,-0.4) node[vertex]{} -- (0,-0.4) node[vertex]{} -- (1,-0.4) node[vertex]{} -- (2,-0.4)node[root]{};
    \draw (-3,0) -- (-2,0.4) node[vertex]{} -- (-1,0.4) node[vertex]{} -- (0,0.4) node[vertex]{} -- (1,0.4) node[vertex]{} -- (2,0.4) node[root]{};
    \draw (-3,0) node[vertex]{} -- (-2,1.2) node[vertex]{} -- (-1,1.2) node[vertex]{} -- (0,1.2) node[vertex]{} -- (1,1.2) node[vertex]{} -- (2,1.2) node[root]{};
    \draw [bracket-r] (2,1.2) -- (2,-0.4) node[black,midway,xshift=10pt]{\footnotesize $s$};
    \draw [bracket-u] (-2,1.2) -- (1,1.2) node[black,midway,yshift=10pt]{\footnotesize $t$};
    \draw [bracket] (0,-1.2) -- (-2,-1.2) node[black,midway,yshift=-10pt]{\footnotesize $t'$};
    \node [below] at (0,-3.2) {$S_{s,t,t'}$};
  \end{tikzpicture}\qquad%
  \begin{tikzpicture}[thick, scale=0.45, baseline=(v.base)]
    \coordinate (v) at (0,0);
    \draw (-1,-1) node[root]{} -- (-1,0);
    \draw (0,-1) node[root]{} -- (0,0);
    \draw (1,-1) node[root]{} -- (1,0);
    \draw (2,-1) node[root]{} -- (2,0);
    \draw (-1,0) node[vertex]{} -- (0,0) node[vertex]{} -- (1,0) node[vertex]{} -- (2,0) node[vertex]{};
    \node [below] at (0.5,-3.2) {$T_{4,7}$};
  \end{tikzpicture}
  \caption{Balanced rooted trees, where $s, t, t'$ refer to vertices, except $t$ in $Q_{s,t}$.} \label{fig:known}
\end{figure}
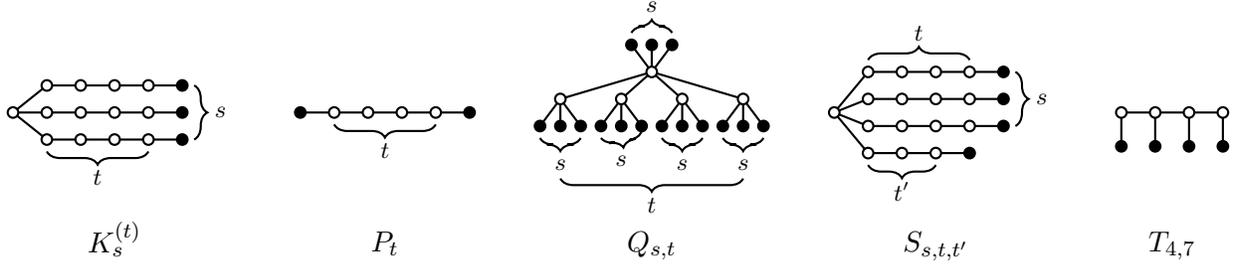

These recent attacks on the Bukh--Conlon conjecture are full of interesting and promising techniques. In this paper, inspired by these previous attempts, we formulate an underlying framework that centers around a notion which we call negligible obstructions (\cref{def:neg,def:obstruction}). In this context, we develop a lemma (\cref{lem:master}), which we call the negligibility lemma, to connect negligible obstructions with the Bukh--Conlon conjecture.
To our best knowledge, ideas in our formulation of the framework can be traced back to the work of Conlon and Lee~\cite{CL19}, and can be spotted throughout later work by various authors.

To establish an instance of the Bukh--Conlon conjecture, the negligibility lemma naturally leads to a two-step strategy: the identification of obstructions and the certification of their negligibility. By no means we claim that this strategy reduces the difficulty of \cref{conj:bc}. Nevertheless we propose this strategy in hopes that it will bring us one step closer to pinning down a handful of essentially different techniques in this area, akin to the theory of flag algebras~\cite{R07}.

We illustrate the above two steps with the proof of \cref{thm:main}. In contrast with all the previous work which has the inductive flavor of certifying negligibility of larger obstructions by that of the smaller, our implementation of the second step has a distinctive inductive pattern, which is elaborated at the end of \cref{sec:idea}. We point out that although \cref{thm:main} can be seen as an extension of \cite[Section~3]{KKL18} which dealt with $Q_{s,t}$, our approach is quite different.

Turning to realizability of rational exponents, our main result \cref{thm:main} gives realizability of the following rational exponents.

\begin{corollary} \label{cor:main}
  For every rational number $r\in (1,2)$ of the form $2 - a/b$, where $a, b \in \N^+$, if 
  \begin{equation} \label{eqn:fractions-2}
    \floor{b/a}^3 \le a \le b/(\floor{b/a} + 1) + 1,
  \end{equation}
  then there exists a bipartite graph $F_r$ such that $\ex(n, F_r) = \Theta(n^r)$.
\end{corollary}

\begin{proof}
  In case $a = 1$, \cref{eqn:fractions-2} forces $b = 1$, which contradicts with the assumption that $r > 1$. Hereafter we assume that $a \ge 2$.
  Now take $s = \floor{b/a}$, $t = a-1$ and $s' = b - (a-1)(\floor{b/a} + 1)$. Set $T = T_{s,t,s'}$. One can easily check that $s, t \in \N^+$, $\rho_T = (st+t+s')/(t+1) = b/a$ and so $\rho_T > 1, \rho_T \ge s$ and $s' \le b - (a-1)b/a = \rho_T$. Observe that \cref{eqn:fractions-2} is equivalent to $t \ge s^3 - 1$ and $s' \ge 0$. In view of \cref{lem:balance-condition}, $T$ is balanced. The corollary follows from \cref{lem:bc-lb,thm:main} immediately.
\end{proof}

As far as we know, all the rationals in $(1,2)$ for which \cref{conj:main} has been verified can be derived from \cref{lem:bc-lb} and the existing instances of \cref{conj:bc}. For convenience, we say a fraction $b/a$ is a \emph{Bukh--Conlon density} if there exists a balanced rooted tree $F$ such that $\rho_F = b/a$ and $\ex(n, F^p) = O(n^{2-1/\rho_F})$ for every $p \in \N^+$. Kang, Kim and Liu observed in \cite[Lemma~4.3]{KKL18} that a graph densification operation due to Erd\H{o}s and Simonovits~\cite{ES70} can be used to generate more Bukh--Conlon densities: whenever $b/a$ is a Bukh--Conlon density, so is $m + b/a$ for every $m \in \N$.

It appears reasonable to restrict our attention to the fractions $b/a$ of the form $m + s/a$ where $m \in \N^+$, for fixed $s, a \in \N$ with $s < a$. The results listed in \cref{fig:known} yield Bukh--Conlon densities $m + s/a$ for every $m \in \N^+$ whenever $s\lceil (a-1)/(s+1) \rceil \le a-1$.\footnote{Combining \cite[Lemma~4.3]{KKL18} with the results listed in \cref{fig:known} (essentially with the one on $S_{s,t,t'}$), we know that $m + s/(st + t' + 1)$ is a Bukh--Conlon density for $m, s \in \N^+$ and $t, t' \in \N$ with $t' \le t$. For $m + s/a$ to be a fraction of such form, one needs $st + 1 \le a \le st + t + 1$ for some $t \in \N$, or equivalently $s\lceil (a-1)/(s+1) \rceil \le a-1$.} For many choices of $(s,a)$, for example $(4,7)$, $(5,8)$ or $(7,10)$, it was not known whether $m + s/a$ is a Bukh--Conlon density for any $m \in \N^+$. For comparison, the family of fractions $b/a$ given by \cref{eqn:fractions-2} generates the Bukh--Conlon densities $m + s/a$ for all $m \ge a - s - 1$ whenever $a - 1 - \sqrt[3]{a} \le s \le a - 1$. In particular, our result gives new Bukh--Conlon densities of the form $m + 5/8$ and $m + 7/10$ as long as $m \ge 2$. Unfortunately our result does not give any Bukh--Conlon densities of the form $m + 4/7$.
The above discussion leads us to the following conjecture on Bukh--Conlon densities.

\begin{conjecture} \label{conj:bc-density}
  For every $s, a \in \N$ with $s < a$, there exists $m \in \N^+$ such that $m + s/a$ is a Bukh--Conlon density.
\end{conjecture}

We point out that one would settle \cref{conj:bc-density} if one could remove the technical condition $t \ge s^3 - 1$ from \cref{thm:main}.

\begin{remark}
  After this work is completed, Conlon and Janzer~\cite{CJ22}, partly building on our ideas, improved \cref{thm:main} by removing the technical condition $t \ge s^3 - 1$, and they hence resolved \cref{conj:bc-density}.
\end{remark}

The rest of the paper is organized as follows. In \cref{sec:idea} we flesh out the aforementioned framework, and use it to prove \cref{thm:main}. In \cref{sec:main} we prove the negligibility lemma. In \cref{sec:star,sec:tree} we certify the negligibility of two different obstructions needed for the proof of \cref{thm:main}.

\section{Negligible obstruction family} \label{sec:idea}

Throughout the rest of the paper, when we view a tree $F$ as a rooted tree, by default the root set $R(F)$ of $F$ consists exactly of the leaves of $F$. We use $V(G)$ and $E(G)$ to denote the vertex set and the edge set of $G$ respectively.

To motivate the relevant concepts, it is instructive to think about finding a copy of $F^p$ in an $n$-vertex $d$-regular graph $G$, where $F$ is a tree and $d = \omega(n^{1-1/\rho_F})$. We mostly talk about embeddings rather than subgraphs.

\begin{definition}[Embedding]
  Given a tree $F$ and a graph $G$, denote $\Inj(F, G)$ the set of \emph{embeddings} from $F$ to $G$, that is, the set of injections $\eta\colon V(F) \to V(G)$ such that $\eta(e) \in E(G)$ for every $e \in E(F)$. For a subset $U$ of $R(F)$ and an injection $\sigma\colon U \to V(G)$, denote the set of embeddings from $F$ to $G$ relativized to $\sigma$ by
  $$
    \Inj(F, G; \sigma) = \dset{\eta \in \Inj(F, G)}{\eta(u) = \sigma(u) \text{ for every }u \in U}.
  $$
  When we write these operators (and the ones coming later) in lowercase, we refer to their cardinalities, for example, $\inj(F, G) = \abs{\Inj(F, G)}$ and $\inj(F, G; \sigma) = \abs{\Inj(F, G; \sigma)}$.
\end{definition}

\begin{remark}
  We encourage the readers who are accustomed to counting subgraphs to think of the embedding counting $\inj(F, G)$ as the corresponding subgraph counting of $F$ in $G$, because they merely differ by a multiplicative factor depending only on $F$. We choose embeddings over subgraphs based on the pragmatic reason that it is more succinct to write in the language of embeddings when counting relativized to some injection $\sigma$.
\end{remark}

Note that $\inj(F, G) \ge \Omega(nd^{e(F)})$ as one can embed $F$ into $G$ one vertex at a time. Because $nd^{e(F)} = \omega(n^{1+e(F)(1-1/\rho_F)}) = \omega(n^{1+e(F)-v(F)+\abs{R(F)}}) = \omega(n^{\abs{R(F)}})$, by the pigeonhole principle, there exists $\sigma \colon R(F) \to V(G)$ such that $\inj(F, G; \sigma) = \omega(1)$.
Ideally the images of $V(F) \setminus R(F)$ under some $p$ embeddings in $\Inj(F, G; \sigma)$ are pairwise (vertex) disjoint, and thus such $p$ embeddings would give us a copy of $F^p$ in $G$. To that end, we define the following notion.

\begin{definition}[Ample embedding] \label{def:ample}
  Given a tree $F$ and a graph $G$, for $\eta \in \Inj(F, G)$, we say $\eta$ is \emph{$C$-ample} if there exist $\eta_1, \dots, \eta_{C} \in \Inj(F, G)$ such that $\eta_i$ and $\eta$ are identical on $R(F)$, and the images of $V(F) \setminus R(F)$ under $\eta_1, \dots, \eta_{C}$ are pairwise disjoint. Given $C \in \N$, denote $\Amp_{C}(F, G)$ the set of $C$-ample embeddings from $F$ to $G$. For a subset $U$ of $R(F)$ and an injection $\sigma\colon U \to V(G)$, the relativized version of $\Amp_{C}(F, G)$, denoted by $\Amp_C(F, G; \sigma)$, is just $\Amp_C(F, G) \cap \Inj(F, G; \sigma)$.
\end{definition}

However it could happen that many embeddings in $\Inj(F, G; \sigma)$ map a nonempty subset of $V(F) \setminus R(F)$ in the same way, thus preventing us from finding a $p$-ample embedding in $\Inj(F, G; \sigma)$. These possible obstructions are encapsulated in the following definitions.

\begin{definition}[Rooted subgraph]
  Given two rooted graphs $F_1$ and $F_2$, we say that $F_2$ contains $F_1$ as a \emph{rooted subgraph} if there exists an embedding $\eta$ from $F_1$ to $F_2$ such that for every $v \in V(F_1)$, $\eta(v) \in R(F_2)$ if and only if $v \in R(F_1)$.
\end{definition}

\begin{definition}[Obstruction family] \label{def:obstruction}
  Given a tree $F$, a family $\F_0$ of trees is an \emph{obstruction family} for $F$ if every member of $\F_0$ is isomorphic to a subtree of $F$ that is not a single edge, and moreover for every nonempty proper subset $U$ of $V(F) \setminus R(F)$, after adding $U$ to the root set of $F$, the resulting rooted graph contains a member of $\F_0$ as a rooted subgraph. (See \cref{fig:cut-tree,lem:obstruction-char} for a concrete example of an obstruction family.)
\end{definition}

\begin{figure}
  \centering
  \begin{tikzpicture}[thick, scale=0.5]
    \fill[gray!60] (-.9,-.4) circle (0.5);
    \fill[gray!60] (.9,-.4) circle (0.5);
    \fill[gray!60] (-.9,-.9)--(.9,-.9)--(.9,.1) --(-.9,.1)--cycle;
    \node at (1.8,-.4) {\footnotesize $U$};

    \draw (-4.5,-.4) -- (-5.1,-2) node[root]{};
    \draw (-4.5,-.4) -- (-4.5,-2) node[root]{};
    \draw (-4.5,-.4) -- (-3.9,-2) node[root]{};
    \draw (0,1.2) -- (-4.5,-.4) node[vertex]{};

    \draw (-2.7,-.4) -- (-3.3,-2) node[root]{};
    \draw (-2.7,-.4) -- (-2.7,-2) node[root]{};
    \draw (-2.7,-.4) -- (-2.1,-2) node[root]{};
    \draw (0,1.2) -- (-2.7,-.4) node[vertex]{};

    \draw (-.9,-.4) -- (-1.5,-2) node[root]{};
    \draw (-.9,-.4) -- (-.9,-2) node[root]{};
    \draw (-.9,-.4) -- (-.3,-2) node[root]{};
    \draw (0,1.2) -- (-.9,-.4) node[vertex]{};

    \draw (.9,-.4) -- (.3,-2) node[root]{};
    \draw (.9,-.4) -- (.9,-2) node[root]{};
    \draw (.9,-.4) -- (1.5,-2) node[root]{};
    \draw (0,1.2) -- (.9,-.4) node[vertex]{};

    \draw (0,1.2) node[vertex]{} -- (2.7,-.4) node[root]{};
    \draw (0,1.2) node[vertex]{} -- (4.5,-.4) node[root]{};
    
    \draw[->] (5.5,-0.4) -- (8.5,-0.4);

    \begin{scope}[shift={(14, 0)}]
      \fill[gray!60] (0,1.2) circle (0.5);
      \fill[gray!60] (-.9,-.4) circle (0.5);
      \fill[gray!60] (4.5,-.4) circle (0.5);
      \fill[gray!60] (-.9,-.9)--(4.5,-.9)--(4.6675,0.0711) --(0.1675,1.6711)--(-0.4358,1.4451)--(-1.3358,-0.1549)--cycle;

      \draw (-4.5,-.4) -- (-5.1,-2) node[root]{};
      \draw (-4.5,-.4) -- (-4.5,-2) node[root]{};
      \draw (-4.5,-.4) -- (-3.9,-2) node[root]{};
      \draw (0,1.2) -- (-4.5,-.4) node[vertex]{};
  
      \draw (-2.7,-.4) -- (-3.3,-2) node[root]{};
      \draw (-2.7,-.4) -- (-2.7,-2) node[root]{};
      \draw (-2.7,-.4) -- (-2.1,-2) node[root]{};
      \draw (0,1.2) -- (-2.7,-.4) node[vertex]{};  
    
      \draw (-.9,-.4) -- (-1.5,-2) node[root]{};
      \draw (-.9,-.4) -- (-.9,-2) node[root]{};
      \draw (-.9,-.4) -- (-.3,-2) node[root]{};
      \draw (0,1.2) -- (-.9,-.4) node[root]{};

      \draw (.9,-.4) -- (1.5,-2) node[root]{};
      \draw (.9,-.4) -- (.9,-2) node[root]{};
      \draw (.9,-.4) -- (.3,-2) node[root]{};
      \draw (0,1.2) -- (.9,-.4) node[root]{};

      \draw (0,1.2) -- (4.5,-.4) node[root]{};
      \draw (0,1.2) node[vertex]{} -- (2.7,-.4) node[root]{};
    \end{scope}
      \end{tikzpicture}
  \caption{After adding $U$ to the root set of $T_{3,4,2}$, the resulting rooted graph contains $K_{1,4}$ as a rooted subgraph.} \label{fig:cut-tree}
\end{figure}
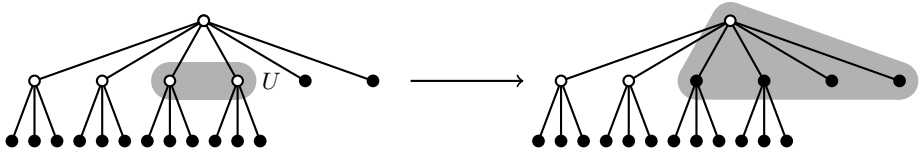

The following definition quantifies the conditions on the obstruction family for $F$ that ensure the existence of a $p$-ample embedding from $F$ to $G$.

\begin{definition}[Negligible obstruction] \label{def:neg}
  Given two trees $F_0$ and $F$, we say that $F_0$ is \emph{negligible} for $F$ if for every $p \in \N^+$ and $\eps > 0$ there exist $c_0 > 0$ and $C_0 \in \N$ such that the following holds. For every $c > c_0$ and every $n$-vertex graph $G$ with $n \ge n_0(c)$, if every vertex in $G$ has degree between $d$ and $Kd$, where $d = cn^\alpha$, $K = 5^{4/\alpha}$ and $\alpha = 1-1/\rho_F$, and moreover $\amp_p(F, G) = 0$, then $\amp_{C_0}(F_0, G) \le \eps nd^{e(F_0)}$. An obstruction family for $F$ is negligible if every member of the family is negligible for $F$.
\end{definition}

\begin{remark}
  As we shall see later in \cref{sec:star,sec:tree}, when certifying the negligibility of an obstruction family, the concrete form of $K$ is unimportant as long as it depends only on $F$. However, since we only need that specific $K$ for \cref{lem:master} to work, we state it explicitly to avoid introducing an additional universal quantifier in \cref{def:neg}.
\end{remark}

We wrap up the above discussion in the following lemma, and we postpone its proof to \cref{sec:main}.

\begin{lemma}[Negligibility lemma] \label{lem:master}
  Given a tree $F$, if there exists a negligible obstruction family $\F_0$ for $F$, then $\ex(n, F^p) = O(n^{2-1/\rho_F})$ for every $p \in \N^+$.
\end{lemma}

The negligibility lemma provides us a two-step strategy to establish \cref{conj:bc} for a balanced rooted tree $F$: first identifying an obstruction family $\F_0$ for $F$, and second certifying the negligibility of $\F_0$. Although in the first step there might be multiple obstruction families for $F$, heuristically speaking it makes more sense to choose $\F_0$ that is minimal under inclusion, because all the heavy lifting happens in the second step that certifies the negligibility of each member of $\F_0$.

Coming back to the tree $T_{s,t,s'}$ defined in \cref{fig:2-tree}, we choose the following obstruction family which is indeed minimal under inclusion.

\begin{proposition} \label{lem:obstruction-char}
  For every $s, t \in \N^+$ and $s' \in \N$, if $(t,s') \neq (1,0)$, then the family $\set{K_{1,s+1}} \cup \dset{T_{s,t-i,s'+i}}{1 \le i \le s - s'}$ is an obstruction family for $T_{s,t,s'}$.
\end{proposition}

\begin{proof}
  Let $F = T_{s,t,s'}$, and let $U$ be a nonempty proper subset of $P \cup Q$, where $P$ and $Q$ are vertex subsets of $V(F)$ defined in \cref{fig:vertex-subsets}. Let $F_+$ be the rooted graph after adding $U$ to the root set $R(F)$ of $F$. If $U$ contains the vertex in $P$, then $F_+$ contains $K_{1,s+1}$ as a rooted subgraph. Otherwise $U \subseteq Q$. In this case, $F_+$ contains $T_{s,t-i,s'+i}$ as a rooted subgraph, where $i = \abs{U}$. Finally notice that when $s' + i \ge s + 1$, $T_{s,t-i,s'+i}$ contains $K_{1,s+1}$ as a rooted subgraph, and so does $F_+$ (see \cref{fig:cut-tree} for an example).
\end{proof}

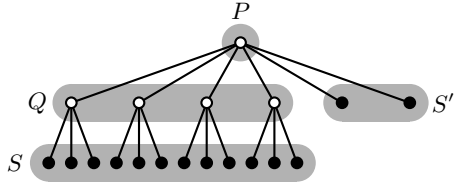
\begin{figure}
  \centering
  \begin{tikzpicture}[thick, scale=0.5]
    \fill[gray!60] (0,1.2) circle (0.5);
    \node at (0,2.05){\footnotesize $P$};

    \fill[gray!60] (2.7,-0.4) circle (0.5);
    \fill[gray!60] (4.5,-0.4) circle (0.5);
    \fill[gray!60] (4.5,-0.9)--(2.7,-0.9)--(2.7,0.1)--(4.5,0.1)-- cycle;
    \node at (5.4,-0.4){\footnotesize $S'$};

    \fill[gray!60] (.9,-0.4) circle (0.5);
    \fill[gray!60] (-4.5,-0.4) circle (0.5);
    \fill[gray!60] (-4.5,-0.9)--(.9,-0.9)--(.9,0.1)--(-4.5,0.1)-- cycle;
    \node at (-5.4,-0.4){\footnotesize $Q$};

    \fill[gray!60] (-5.1,-2) circle (0.5);
    \fill[gray!60] (1.5,-2) circle (0.5);
    \fill[gray!60] (-5.1,-2.5)--(1.5,-2.5)--(1.5,-1.5)--(-5.1,-1.5)-- cycle;
    \node at (-6,-2) {\footnotesize $S$};
    
    \draw (-4.5,-.4) -- (-5.1,-2) node[root]{};
    \draw (-4.5,-.4) -- (-4.5,-2) node[root]{};
    \draw (-4.5,-.4) -- (-3.9,-2) node[root]{};
    \draw (0,1.2) -- (-4.5,-.4) node[vertex]{};

    \draw (-2.7,-.4) -- (-3.3,-2) node[root]{};
    \draw (-2.7,-.4) -- (-2.7,-2) node[root]{};
    \draw (-2.7,-.4) -- (-2.1,-2) node[root]{};
    \draw (0,1.2) -- (-2.7,-.4) node[vertex]{};

    \draw (-.9,-.4) -- (-1.5,-2) node[root]{};
    \draw (-.9,-.4) -- (-.9,-2) node[root]{};
    \draw (-.9,-.4) -- (-.3,-2) node[root]{};
    \draw (0,1.2) -- (-.9,-.4) node[vertex]{};

    \draw (.9,-.4) -- (.3,-2) node[root]{};
    \draw (.9,-.4) -- (.9,-2) node[root]{};
    \draw (.9,-.4) -- (1.5,-2) node[root]{};
    \draw (0,1.2) -- (.9,-.4) node[vertex]{};

    \draw (0,1.2) -- (2.7,-.4) node[root]{};
    \draw (0,1.2) node[vertex]{} -- (4.5,-.4) node[root]{};
  \end{tikzpicture}
  \caption{Vertex partition of $T_{s,t,s'}$.} \label{fig:vertex-subsets}
\end{figure}

\cref{thm:main} follows immediately from the next theorem which certifies the negligibility of the obstruction family defined in \cref{lem:obstruction-char}.

\begin{theorem} \label{thm:neg}
  For every $s,t \in \N^+$ and $s' \in \N$, when $s - s' \ge 2$, assume in addition that
  \begin{equation} \label{eqn:t-cond}
    t \ge \left(1-\tfrac{s'}{s+1}\right)k\left(k-\tfrac{1}{s}\right)(s+2-k)+\tfrac{1}{s}, \quad\text{for every } 2 \le k \le s-s'.
  \end{equation}
  If $T := T_{s,t,s'}$ is balanced, then for every $p \in \N^+$ and $\eps > 0$, there exists $c_0 > 0$ such that the following holds.
  For every $c > c_0$ and every $n$-vertex graph $G$ with $n \ge n_0(c)$, if every vertex in $G$ has degree between $d$ and $Kd$, where $d = cn^\alpha$, $K = 5^{4/\alpha}$ and $\alpha = 1 - 1/\rho_T$, and moreover $\amp_p(T, G) = 0$, then
  \begin{enumerate}[label=(\alph*)]
    \item $\amp_{C_*}(K_{1,s+1}, G) \le \eps nd^{e(K_{1,s+1})}$, where $C_* = v(T_{s,t,s+1}^p)$; and \label{item:neg-1}
    \item $\amp_{C_k}(F_k, G) \le \eps nd^{e(F_k)}$, where $C_k = pv(T)^k$ and $F_k = T_{s,t-k,s'+k}$, for every $1 \le k \le s - s'$. \label{item:neg-2}
  \end{enumerate}
\end{theorem}

\begin{proof}[Proof of \cref{thm:main}]
  Suppose that $T := T_{s,t,s'}$ is balanced. When $s \le s'$, the obstruction family for $T$ consists of a single $K_{1,s+1}$, which by \cref{thm:neg}\ref{item:neg-1} is negligible for $T$. When $s - s' = 1$, in view of \cref{thm:neg}, the obstruction family $\F_0$ defined in \cref{lem:obstruction-char} is also negligible. When $s - s' \ge 2$, $\F_0$ is negligible provided \cref{eqn:t-cond}. Observe that $t \ge s^3 - 1$ ensures \cref{eqn:t-cond}. Indeed, the right hand side of \cref{eqn:t-cond} is at most $k^2(s+2-k)+1/s$, which, by the inequality of arithmetic and geometric means, is at most $(2(s+2)/3)^3/2+1/s$, which is at most $s^3-1$ for $s \ge 3$. One can check directly in case $s = 2$ that the right hand side of \cref{eqn:t-cond} is less than $7$. In any case, it then follows from \cref{lem:master} that $\ex(n, T^p) = O(n^{2-1/\rho_F})$ for all $p \in \N^+$.
\end{proof}

Our proof of \cref{thm:neg} is inductive in nature. In \cref{sec:star} we first establish the negligibility of $K_{1,s+1}$ in \cref{thm:neg}\ref{item:neg-1}. In \cref{sec:tree} we deduce the negligibility of $F_k$ in \cref{thm:neg}\ref{item:neg-2} from that of $K_{1,s+1}$ and $F_1, \dots, F_{k-1}$. The inductive pattern here is counterintuitive in the sense that the negligibility of $F_k$, which is a subgraph of $F_{k-1}$, comes after that of $F_{k-1}$.

\section{Proof of the negligibility lemma} \label{sec:main}

In \cref{sec:idea}, we have analyzed the special case where the graph $G$ is regular. In the context of degenerate extremal graph theory, it is indeed standard to assume that $G$ is almost regular. This idea due to Erd\H{o}s and Simonovits first appeared in \cite{ES70}. We shall use the following variant (see also \cite[Proposition~2.7]{JS12} for a similar result).

\begin{lemma}[Theorem 12 of Bukh and Jiang \cite{BJ17}, only in arXiv version] \label{lem:folk-reg}
  For every $c > 0$ and $\alpha \in (0,1]$, there exists $\tilde{n}_0 \in \N$ such that the following holds. Every $\tilde{n}$-vertex graph with $\tilde{n} \ge \tilde{n}_0$ and at least $(6c/\alpha)\tilde{n}^{1+\alpha}$ edges contains an $n$-vertex subgraph $G$ with $n \ge (6c/\alpha)\tilde{n}^{\alpha/2}$ such that every vertex in $G$ has degree between $cn^\alpha$ and $Kcn^\alpha$, where $K = 5^{4/\alpha}$. \qed
\end{lemma}

We now formalize the discussion in \cref{sec:idea} on finding a copy of $F^p$ in $G$.

\begin{definition}[Extension] \label{def:extension}
  Given two trees $F_1, F_2$ and a graph $G$, for $\eta_1 \in \Inj(F_1, G)$ and $\eta_2 \in \Inj(F_2, G)$, we say $\eta_2$ \emph{extends} $\eta_1$ if $\eta_1 = \eta_2 \circ \eta_{12}$ for some embedding $\eta_{12}\in\Inj(F_1, F_2)$. Given $C \in \N$, denote $$\Ext_C(F_1, F_2, G) = \dset{\eta_2 \in \Inj(F_2, G)}{\eta_2 \text{ extends some }\eta_1 \in \Amp_C(F_1, G)}.$$
\end{definition}

\begin{proof}[Proof of \cref{lem:master}]
  Suppose that $F$ is a tree, $p \in \N^+$ and $\F_0$ is a negligible obstruction family for $F$. Let $c > 0$ be a constant to be determined later. We would like prove that $\ex(\tilde{n}, F^p) < (6/\alpha)c\tilde{n}^{1+\alpha}$ for all $\tilde{n} \ge \tilde{n}_0(c)$, where $\alpha = 1 - 1/\rho_F$. By \cref{lem:folk-reg}, it suffices to prove that every $n$-vertex graph $G$ with $n \ge n_0(c)$, if every vertex in $G$ has degree between $cn^\alpha$ and $Kcn^\alpha$, where $K = 5^{4/\alpha}$, then $G$ contains $F^p$ as a subgraph.
  
  Suppose that $G$ is an $n$-vertex graph with $n \ge n_0(c)$ such that every vertex in $G$ has degree between $d$ and $Kd$, where $d = cn^\alpha$. For the sake of contradiction, we assume that $\amp_p(F, G) = 0$. With hindsight, take
  $$
    \eps = \frac{K^{-e(F)}}{3\sum_{F_0 \in \F_0}\inj(F_0, F)}.
  $$
  Unwinding \cref{def:neg}, we obtain two constants $c_{F_0} > 0$ and $C_{F_0} \in \N$ for every $F_0 \in \F_0$. If we had chosen $c \ge \max \dset{c_{F_0}}{F_0 \in \F_0}$, then for every $F_0 \in \F_0$, $\amp_{C_{F_0}}(F_0, G) \le \eps n d^{e(F_0)}$, and in particular, $\amp_{C_0}(F_0, G) \le \eps n d^{e(F_0)}$, where $C_0 = \max(\dset{C_{F_0}}{F_0\in \F_0} \cup \set{p})$.

  Consider the embeddings in
  \begin{equation} \label{eqn:exclude-ext}
    I := \Inj(F, G) \setminus \bigcup_{F_0 \in \F_0}\Ext_{C_0}(F_0, F, G).
  \end{equation}
  Clearly $\inj(F, G) \ge (1-o(1))nd^{e(F)}$, and moreover for every $F_0 \in \F_0$,
  $$
    \ext_{C_0}(F_0, F, G) \le \inj(F_0, F)\amp_{C_0}(F_0, G)(Kd)^{e(F)-e(F_0)} \le \eps \inj(F_0, F) K^{e(F)} nd^{e(F)}.
  $$
  We can estimate the cardinality of $I$ by
  $$
    \abs{I} \ge \left(1-o(1)\right)nd^{e(F)} - \eps\sum_{F_0 \in \F_0}\inj(F_0, F)K^{e(F)}nd^{e(F)} = (2/3-o(1)) nd^{e(F)},
  $$
  and so $\abs{I} \ge nd^{e(F)}/2 = c^{e(F)}n^{1 + e(F)(1-1/\rho_F)}/2 = c^{e(F)}n^{\abs{R(F)}}/2$ if we had chosen $n_0(c)$ large enough.
  
  By the pigeonhole principle, the cardinality of $I_\sigma := I \cap \Inj(F, G; \sigma)$ is at least $c^{e(F)}/2$ for some $\sigma \colon R(F) \to V(G)$. For every $U \subseteq V(F) \setminus R(F)$ and every injection $\tau\colon U\to V(G)$, set $$I_\sigma(\tau) = \dset{\eta \in I_\sigma}{\eta(u) = \tau(u) \text{ for every }u \in U}.$$
  
  \begin{claim*}
    For every $U \subseteq V(F)\setminus R(F)$ and $\tau\colon U\to V(G)$, $$\abs{I_\sigma(\tau)} \le (C_0v(F)^2)^{v(F) - \abs{R(F)} - \abs{U}}.$$
  \end{claim*}

  \begin{claimproof}
    We prove by backward induction on $\abs{U}$. Clearly $\abs{I_\sigma(\tau)} \le 1$ when the domain $U$ of $\tau$ equals $V(F) \setminus R(F)$. Suppose $U$ is a proper subset of $V(F) \setminus R(F)$. Recall from \cref{def:obstruction} that after adding $U$ to the root set of $F$, the resulting rooted graph contains $F_0$ as a rooted subgraph that is isomorphic to a member of $\F_0$. Notice that $U_0 := V(F_0) \setminus R(F_0)$ is nonempty because $F_0$ is not a single edge.
    
    Let $I_\sigma'(\tau)$ be a maximal subset of $I_\sigma(\tau)$ such that the images of $U_0$ under the embeddings in $I'_\sigma(\tau)$ are pairwise disjoint, and let $V_0$ be the union of these images. Since $I_\sigma(\tau) \subseteq I$ and $I$ defined by \cref{eqn:exclude-ext} contains no extension of any $C_0$-ample embedding from $F_0$ to $G$, we bound $\abs{I'_\sigma(\tau)} < C_0$, which implies that $\abs{V_0} < C_0\abs{U_0}$. For each $u \in U_0$ and $v \in V_0$, by the inductive hypothesis $$
      \abs{I_\sigma(\tau_{uv})} < (C_0v(F)^2)^{v(F)-\abs{R(F)}-\abs{U}-1},
    $$
    where $\tau_{uv}\colon U \cup \set{u} \to V(G)$ extends $\tau$ by mapping $u$ to $v$ additionally. The maximality of $I'_\sigma(\tau)$ means that for every $\eta \in I_\sigma(\tau)$ there is $u \in U_0$ such that $\eta(u) \in V_0$, and so $\eta \in I_\sigma(\tau_{uv})$ for some $v \in V_0$. Therefore $$\abs{I_\sigma(\tau)} \le \sum_{u\in U_0, v\in V_0}\abs{I_\sigma(\tau_{uv})} < \abs{U_0}\abs{V_0}(C_0v(F)^2)^{v(F)-\abs{R(F)}-\abs{U}-1},$$
    which implies the inductive step as $\abs{U_0} < v(F)$ and $\abs{V_0} < C_0\abs{U_0}$.

    The same argument works for the last inductive step where $U = \varnothing$ because there is no $p$-ample embedding from $F$ to $G$, and $C_0 \ge p$.
  \end{claimproof}
  
  In particular, $I_\sigma = I_\sigma(\tau)$ when the domain of $\tau$ is an empty set, and so $\abs{I_\sigma} \le (C_0v(F)^2)^{v(F)-\abs{R(F)}}$, which would yield a contradiction if we had chosen $c > (2(C_0v(F)^2)^{v(F)-\abs{R(F)}})^{1/e(F)}$.
\end{proof}

\section{Ample embeddings of stars} \label{sec:star}

The negligibility of $K_{1,s+1}$ for $T_{s,t,s'}$ is established directly through the following technical lemma.

\begin{lemma} \label{lem:ample-star}
  For every $s, t \in \N^+$ and $s' \in \N$, set $s_0 = \max(s', 1)$, $F_0 = K_{1,s_0}$, $F_1 = K_{1,s+1}$ and $T = T_{s,t,s'}$. For every $p \in \N^+$ and $\eps > 0$, there exists $c_0 > 0$ such that  for every $n$-vertex graph $G$, if $\amp_p(T, G) = 0$ and $\inj(F_0, G) \ge c_0n^{s_0}$, then $\amp_{C_1}(F_1, G) \le \eps\inj(F_1, G)$, where $C_1 = v(T_{s,t,s_0}^p)$.
\end{lemma}

Our proof of \cref{lem:ample-star} follows the outline of \cite[Lemma~5.3]{CJL19}. Over there the conclusion, in our language, is that for every $\eps > 0$ there exists $C_1 \in \N$ such that $\amp_{C_1}(F_1, G) \le \eps\inj(F_1, G)$. One can work out the quantitative dependency $C_1 = \Omega(\eps^{-1/(s-1)})$ from their argument. Although this dependency alone is enough for the negligibility of $K_{1,s+1}$, it becomes inadequate when we iteratively apply this bound later in \cref{sec:tree}. To decouple $C_1$ from $\eps$ in \cref{lem:ample-star}, we need the following classical result in degenerate extremal hypergraph theory.

\begin{theorem}[Erd\H{o}s~\cite{E64}] \label{lem:r-partite-turan}
  For every $r$-partite $r$-uniform hypergraph $H$ there exists $\eps > 0$ so that $\ex(n, H) = O(n^{r-\eps})$.%
  \footnote{Given an $r$-uniform hypergraph $H$, the Tur\'an number $\ex(n, H)$ is the maximum number of hyperedges in an $r$-uniform hypergraph on $n$ vertices that contains no $H$ as a subhypergraph.}\qed
\end{theorem}

\begin{proof}[Proof of \cref{lem:ample-star}]
  Suppose that $G$ is an $n$-vertex graph such that $\amp_p(T, G) = 0$ and $\inj(F_0, G) \ge c_0n^{s_0}$, where $c_0$ is to be chosen. As we only deal with embeddings to $G$ in the following proof, we omit $G$ in $\Inj(\cdot, G),\Amp_\cdot(\cdot, G)$ and their relativized versions.
  
  Recall $s_0 = \max(s',1)$. Clearly $G$ contains no $F^p$ as a subgraph, where $F = T_{s,t,s_0}$.
  Let $U_0$ denote an arbitrary vertex subset of size $s_0$ in $G$, and denote $N_G(U_0)$ the common neighborhood of $U_0$ in $G$. Let $H$ be the $(s+1)$-uniform hypergraph on $V(G)$ given by $$H = \dset{\eta(R(F_1))}{\eta \in \Amp_{C_1}(F_1)},$$
  where $C_1 = v(T_{s,t,s_0}^p)$, and denote $H[N_G(U_0)]$ the subhypergraph of $H$ induced on $N_G(U_0)$.

  The strategy is to use $\sum_{U_0}e(H[N_G(U_0)])$ and $\sum_{U_0} \babsng$ as intermediaries to connect $\amp_{C_1}(F_1)$ and $\inj(F_1)$.

  \setcounter{claim}{0}

  \begin{claim} \label{claim:no-hstp}
    There exists $n_0 = n_0(s,t,p,C_1) \in \N$ such that for every $U_0$ with $\absng \ge n_0$,
    $$
      e(H[N_G(U_0)]) \le
      \frac{\eps}{4s_0^{s_0}}\babsng.
    $$
  \end{claim}

  \begin{claimproof}[Proof of \cref{claim:no-hstp}]
    Recall the vertex partition $V(F) = P \cup Q \cup S \cup S'$ from \cref{fig:vertex-subsets}. This partition induces the vertex partition $V(F^p) = \widetilde{P} \cup \widetilde{Q} \cup S \cup S'$, where $\widetilde{P}$ denotes the union of the $p$ disjoint copies of $P$ in $F^p$, and $\widetilde{Q}$ is defined similarly.
    Let $H_0$ be the $(s+1)$-uniform hypergraph on $\widetilde{P} \cup S$ with each hyperedge given by the $s+1$ neighbors of a vertex of $\widetilde{Q}$ in $F^p$.

    Observe that $H[N_G(U_0)]$ never contains $H_0$ as a subhypergraph. Suppose on the contrary that there exists an embedding $\eta$ from $H_0$ to $H[N_G(U_0)]$,%
    \footnote{Given two hypergraphs $H_1$ and $H_2$ of the same uniformity, an embedding from $H_1$ to $H_2$ is just an injection $\eta\colon V(H_1) \to V(H_2)$ such that $\eta(e) \in H_2$ for every $e \in H_1$.}
    then we can embed $F^p$ in $G$ by mapping $S'(F)$ to $U_0$, mapping $P(F^p) \cup S(F)$ according to $\eta$, and embedding the vertices in $Q(F^p)$ greedily. The last step of the embedding is possible because for every hyperedge $e \in H_0$, $\eta(e) = \eta'(R(F_1))$ for some $\eta'\in \Amp_{C_1}(F_1)$, and more importantly $C_1 \ge v(F^p)$.

    Since $H_0$ is an $(s+1)$-partite hypergraph, the claim follows from \cref{lem:r-partite-turan} immediately.
  \end{claimproof}

  We choose such $n_0 \in \N$ in \cref{claim:no-hstp} and require in addition that $n_0 \ge s+1$. For convenience, set
  $$\U = \dset{U_0 \subseteq V(G)}{\abs{U_0} = s_0, \absng \ge n_0}.$$

  \begin{claim} \label{claim:ample-star}
    The number of $C_1$-ample embeddings from $F_1$ to $G$ satisfies
    $$
      \amp_{C_1}(F_1) \le \frac{s_0^{s_0}(s+1)!}{C_1^{s_0-1}}\sum_{U_0} e(H[N_G(U_0)]).
    $$
  \end{claim}

  \begin{claimproof}[Proof of \cref{claim:ample-star}]
    Let $\sigma$ denote an arbitrary injection from $R(F_1)$ to $V(G)$, and denote for short $a(\sigma) = \amp_{C_1}(F_1; \sigma)$. Note that $a(\sigma)$ has the dichotomy that either $a(\sigma) = 0$ or $a(\sigma) \ge C_1 \ge s_0$, which implies that $\binom{a(\sigma)}{s_0} \ge (a(\sigma)/s_0)^{s_0} \ge C_1^{s_0-1}a(\sigma)/s_0^{s_0}$ in either case. Through counting in two ways the disjoint union of the edge sets of $H[N_G(U_0)]$ for all vertex subsets $U_0$ of size $s_0$ in $G$, one can show that
    $$
      (s+1)!\sum_{U_0} e(H[N_G(U_0)]) = \sum_{\sigma} \binom{a(\sigma)}{s_0} \ge \frac{C_1^{s_0-1}}{s_0^{s_0}}\sum_{\sigma} a(\sigma) = \frac{C_1^{s_0-1}}{s_0^{s_0}}\amp_{C_1}(F_1),
    $$
    which implies the desired inequality in the claim.
  \end{claimproof}

  \begin{claim} \label{claim:ample-star-2}
    The number of embeddings from $F_1$ to $G$ satisfies
    $$
      \inj(F_1) \ge \frac{(s+1)!}{2C_1^{s_0}}\sum_{U_0\in \U}\babsng.
    $$
  \end{claim}
  \begin{claimproof}[Proof of \cref{claim:ample-star-2}]
    We count in two ways the disjoint union $\bigsqcup_{U_0 \in \U}I(U_0)$, where
    $$
      I(U_0) := \dset{\eta \in \Inj(F_1) \setminus \Amp_{C_1}(F_1)}{\eta(R(F_1)) \subseteq N_G(U_0)}.
    $$
    On the one hand, for a fixed $U_0$ with $\absng \ge n_0$, every subset of $N_G(U_0)$ of size $s+1$ that is not a hyperedge of $H[N_G(U_0)]$ gives rise to at least $s_0(s+1)!$ many $\eta\in I(U_0)$, and it follows from \cref{claim:no-hstp} that $e(H[N_G(U_0)]) \le \frac12\babsng$. Thus we get
    $$
      \abs{I(U_0)} \ge \frac{s_0(s+1)!}{2} \babsng, \quad\text{for every }U_0 \in \U.
    $$
    On the other hand, for every $\eta\in \Inj(F_1) \setminus \Amp_{C_1}(F_1)$, there are at most $\binom{C_1}{s_0}$ many $U_0$ such that $\eta(R(F_1)) \subseteq N_G(U_0)$. Hence
    $$
      \inj(F_1) \ge \inj(F_1) - \amp_{C_1}(F_1) \ge \frac{1}{\binom{C_1}{s_0}}\sum_{U_0}\abs{I(U_0)} \ge \frac{s_0!}{C_1^{s_0}}\sum_{U_0}\abs{I(U_0)},
    $$
    which implies the desired inequality in the claim.
  \end{claimproof}

  A simple double counting argument shows that 
  $$
    \inj(F_0) = s_0!\sum_{U_0}\absng.
  $$
  Recall the assumption that $\inj(F_0) \ge c_0n^{s_0}$. Thus the average $\bar{N}$ of $\absng$ satisfies $$
    \bar{N} = \frac{\inj(F_0)}{s_0!\binom{n}{s_0}} \ge c_0.
  $$
  We can choose $c_0 > 0$ large enough so that $\binom{\bar{N}}{s+1} \ge \left(1+4s_0^{s_0}C_1/\eps\right)\binom{n_0}{s+1}$.
  By Jensen's inequality, we have
  $$
    \sum_{U_0}\babsng \ge \binom{n}{s_0}\binom{\bar{N}}{s+1} \ge \left(1+4s_0^{s_0}C_1/\eps\right)\sum_{U_0 \not\in \U}\babsng,
  $$
  which implies that
  $$
    \sum_{U_0 \not\in \U} \babsng \le \frac{\eps}{4s_0^{s_0}C_1} \sum_{U_0 \in \U}\babsng.
  $$
  
  Applying \cref{claim:ample-star} and then \cref{claim:no-hstp}, we get
  \begin{multline*}
    \frac{C_1^{s_0-1}}{s_0^{s_0}(s+1)!}\amp_{C_1}(F_1) \le 
    \sum_{U_0} e(H[N_G(U_0)]) \\
    \le \sum_{U_0 \not\in \U}\babsng + \frac{\eps}{4s_0^{s_0}C_1}\sum_{U_0 \in \U}\babsng \le \frac{\eps}{2s_0^{s_0}C_1}\sum_{U_0 \in \U}\babsng,
  \end{multline*}
  which implies
  $$
    \amp_{C_1}(F_1) \le \frac{(s+1)!\eps}{2C_1^{s_0}}\sum_{U_0\in\U}\babsng.
  $$
  Comparing it with \cref{claim:ample-star-2}, we get the desired inequality in \cref{lem:ample-star}.
\end{proof}

\begin{proof}[Proof of \cref{thm:neg}\ref{item:neg-1}]
  For $s,t \in \N^+$ and $s' \in \N$, set $s_0 = \max(s',1)$, and $T = T_{s,t,s'}$. Since $T$ is balanced, by \cref{lem:balance-condition}, $s_0 \le s+1$ and $\rho_T \ge s_0$, the latter of which implies that $1+s_0\alpha \ge s_0$, where $\alpha = 1 - 1/\rho_T$.

  Let $p \in \N^+, C_* = v(T_{s,t,s+1}^p) \ge v(T_{s,t,s_0}^p)$ and $\eps > 0$, and let $c_0 > 0$ be a constant to be determined later. Suppose that $c > c_0$ and $G$ is an $n$-vertex graph with $n \ge n_0(c)$ such that every vertex in $G$ has degree between $d$ and $Kd$, where $d = cn^\alpha$ and $K = 5^{4/\alpha}$, and moreover $\amp_p(T, G) = 0$. Clearly, $\inj(K_{1,s+1}, G) \le n(Kd)^{s}$.
  We apply \cref{lem:ample-star} and obtain $c_1 > 0$ so that if $\inj(K_{1,s_0}, G) \ge c_1n^{s_0}$ then
  $$
    \amp_{C_*}(K_{1,s+1}, G) \le \eps \inj(K_{1,s+1},G) \le \eps n(Kd)^{s+1} = \eps K^{s+1} nd^{e(K_{1,s+1})}.
  $$
  Since $1+s_0\alpha \ge s_0$, we have
  $$
    \inj(K_{1,s_0}, G) \ge (1-o(1))nd^{s_0} = (1-o(1))c^{s_0}n^{1+s_0\alpha} \ge (1-o(1))c^{s_0}n^{s_0}.
  $$
  Thus the condition $\inj(K_{1,s_0}, G) \ge c_1n^{s_0}$ can be met by choosing $c_0 = c_1^{1/s_0}$ and $n_0(c)$ sufficiently large.
\end{proof}

\section{Ample embeddings of subtrees} \label{sec:tree}

\subsection{Preliminary propositions}

For the proof of \cref{thm:neg}\ref{item:neg-2}, we need the following variation of the classical sunflower lemma for sequences (see \cite{ALWZ20} for the recent breakthrough on the sunflower conjecture of Erd\H{o}s and Rado~\cite{ER60} and related background).

\begin{definition}[Sequential sunflower]
  Suppose that $W \subseteq V^k$ is a system of sequences. A subset $S$ of $W$ is a \emph{sequential sunflower} with \emph{kernel} $I \subsetneq [k]$ if for every pair of distinct sequences $(s_1, \dots, s_k), (s_1', \dots, s_k') \in S$, the subsequences $(s_i)_{i \in I}$ and $(s_i')_{i \in I}$ are equal, but the sets $\dset{s_i}{i \not\in I}$ and $\dset{s_i'}{i \not\in I}$ are disjoint.
\end{definition}

\begin{proposition} \label{lem:sunflower}
  Fix $k, C \in \N^+$. Suppose that $W \subseteq V^k$ is a system of sequences such that each sequence in $W$ consists of $k$ distinct elements. If $W$ contains no sequential sunflower of size $C$, then $\abs{W} < (k!)^2(k!C-1)^k$.
\end{proposition}

\begin{proof}
  Consider the system $F$ of subsets of $V$ defined by
  $$
    F = \dset{\set{s_1, \dots, s_k}}{(s_1, \dots, s_k) \in W}.
  $$
  Clearly $\abs{W} \le k!\abs{F}$. We claim that $F$ contains no sunflower of size $k!C$. Recall that a sunflower is a collection of sets whose pairwise intersection is constant. Assuming the claim, the classical sunflower lemma precisely states that $\abs{F} < k!(k!C-1)^k$, which implies the desired inequality. Suppose on the contrary that $E \subseteq F$ is a sunflower of size $k!C$ with kernel $K$. Consider the subsystem of sequences $W_0 = \dset{(s_1, \dots, s_k) \in W}{\set{s_1, \dots, s_k} \in E}$. Clearly $\abs{W_0} \ge k!C$. By the pigeonhole principle, there exist a set $W_1 \subseteq W_0$ of size $C$ and $I \subsetneq [k]$ such that for every $s \in W_1$, $\dset{s_i}{i \in I} = K$ and $(s_i)_{i\in I}$ is a constant subsequence. As $E$ is a sunflower, one can check that $W_1$ is a sequential sunflower of size $C$, which is a contradiction.
\end{proof}

We also need the following classical theorem due to K\H{o}v\'ari, S\'os and Tur\'an~\cite{KST54} on the Zarankiewicz problem.

\begin{proposition} \label{lem:kst}
  Fix $s, t \in \N^+$. Suppose that $H$ is a bipartite graph with two parts $U$ and $W$ such that every vertex in $W$ has degree at least $s$. If $H$ contains no complete bipartite subgraph with $s$ vertices in $U$ and $t$ vertices in $W$, then $e(H) \le K\abs{U}\abs{W}^{1-1/s}$, where $K = s\sqrt[s]{(t-1)/s!}$. \qed
\end{proposition}

The following result is a generalization of a result due to F\"{u}redi~\cite{F91}. Our proof of the generalization follows the proof of F\"{u}redi's result by Alon, Krivelevich, and Sudakov~\cite{AKS03} using dependent random choice (see \cite{FS11} for a survey on dependent random choice). We denote $d_G(v)$ the degree of a vertex $v$ in $G$.

\begin{proposition} \label{lem:drc}
  Fix $k, r\in \N^+$ such that $k < r$. Suppose that $F$ is a bipartite graph with two parts $U_0$ and $W_0$ such that every vertex in $W_0$ has degree at most $r$. For every bipartite graph $G$ with two parts $U$ and $W$, if there is no embedding $\eta$ from $F$ to $G$ such that $\eta(U_0) \subseteq U$ and $\eta(W_0) \subseteq W$, then
  $$
    \sum_{u \in U}d_G(u)^k \le \left(K_1\abs{U}^{k} + K_2\abs{W}^k\right)\abs{U}^{1-k/r},
  $$
  where $K_1 = \abs{W_0}^k/(r!)^{k/r}$ and $K_2 = (\abs{U_0}-1)^{k/r}$.
\end{proposition}

\begin{proof}
  Assume for the sake of contradiction that
  $$
    \sum_{u \in U}d(u)^k > (r!)^{-k/r}\abs{W_0}^k\abs{U}^{k+1-k/r} + (\abs{U_0}-1)^{k/r}\abs{U}^{1-k/r}\abs{W}^k.
  $$
  Pick a subset $W_1 \subseteq W$ of size $r$ uniformly at random with replacement. Set $U(W_1) \subseteq U$ to be the common neighborhood of $W_1$ in $G$, and let $X$ denote the cardinality of $U(W_1)$. By linearity of expectation and H\"older's inequality,
  \begin{multline*}
    \EE[X] = \sum_{u \in U}\left(\frac{d(u)}{\abs{W}}\right)^r \ge \frac{\left(\sum_{u \in U}d(u)^k\right)^{r/k}}{\abs{U}^{r/k-1}\abs{W}^r} \\
    > \frac{(r!)^{-1}\abs{W_0}^r\abs{U}^{r+r/k-1}+(\abs{U_0}-1)\abs{U}^{r/k-1}\abs{W}^r}{\abs{U}^{r/k-1}\abs{W}^r} \ge \frac{\abs{U}^r}{r!}\left(\frac{\abs{W_0}}{\abs{W}}\right)^r + \abs{U_0} - 1.
  \end{multline*}
  Let $Y$ denote the random variable counting the number of subsets $S \subseteq U(W_1)$ of size $r$ with fewer than $\abs{W_0}$ common neighbors in $G$. For a given such $S$, the probability that it is a subset of $U(W_1)$ is less than $(\abs{W_0}/\abs{W})^r$. Since there are at most $\binom{\abs{U}}{r}$ subsets $S$ of size $r$, it follows that
  $$
    \EE[Y] < \binom{\abs{U}}{r}\left(\frac{\abs{W_0}}{\abs{W}}\right)^r \le \frac{\abs{U}^r}{r!}\left(\frac{\abs{W_0}}{\abs{W}}\right)^r.
  $$
  By linearity of expectation,
  $$
    \EE[X - Y] > \frac{\abs{U}^r}{r!}\left(\frac{\abs{W_0}}{\abs{W}}\right)^r + \abs{U_0} - 1 - \frac{\abs{U}^r}{r!}\left(\frac{\abs{W_0}}{\abs{W}}\right)^r = \abs{U_0} - 1.
  $$
  Hence there exists a choice of $W_1$ for which $X - Y \ge \abs{U_0}$. Delete one vertex from each subset $S$ of $U(W_1)$ of size $r$ with fewer than $m$ common neighbors. We let $U'$ be the remaining subset of $U(W_1)$. The set $U' \subseteq U$ has at least $\abs{U_0}$ vertices, and every subset of $U'$ of size $r$ has at least $\abs{W_0}$ common neighbors. One can then greedily find an embedding $\eta$ from $F$ to $G$ such that $\eta(U_0) \subseteq U'$ and $\eta(W_0) \subseteq W$.
\end{proof}

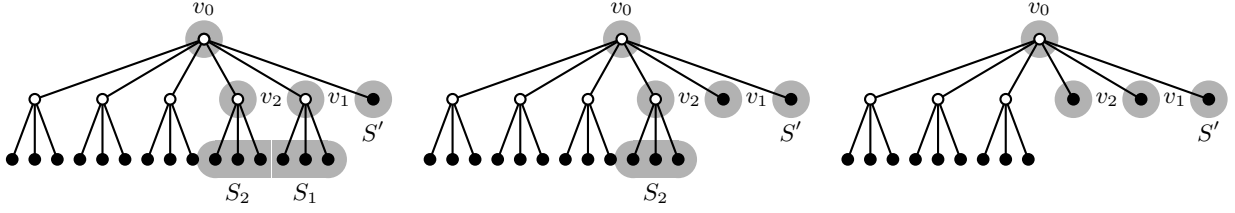
\begin{figure}
  \centering
  \begin{tikzpicture}[thick, scale=0.5, baseline=(v.base)]
    \fill[gray!60] (0,1.2) circle (0.5);
    \node at (0,2) {\footnotesize $v_0$};

    \coordinate (v) at (0,0);
    \fill[gray!60] (.9,-.4) circle (0.5);
    \fill[gray!60] (2.7,-.4) circle (0.5);
    \fill[gray!60] (4.5,-.4) circle (0.5);
    \fill[gray!60] (.3,-2) circle (0.5);
    \fill[gray!60] (3.3,-2) circle (0.5);
    \fill[gray!60] (.3,-2.5)--(1.78,-2.5)--(1.78,-1.5)--(.3,-1.5)--cycle;
    \fill[gray!60] (3.3,-2.5)--(1.82,-2.5)--(1.82,-1.5)--(3.3,-1.5)--cycle;
    \node at (.9,-2.9) {\footnotesize $S_2$};
    \node at (2.7,-2.9) {\footnotesize $S_1$};
    \node at (4.5,-1.25) {\footnotesize $S'$};

    \draw (-4.5,-.4) -- (-5.1,-2) node[root]{};
    \draw (-4.5,-.4) -- (-4.5,-2) node[root]{};
    \draw (-4.5,-.4) -- (-3.9,-2) node[root]{};
    \draw (0,1.2) -- (-4.5,-.4) node[vertex]{};

    \draw (-2.7,-.4) -- (-3.3,-2) node[root]{};
    \draw (-2.7,-.4) -- (-2.7,-2) node[root]{};
    \draw (-2.7,-.4) -- (-2.1,-2) node[root]{};
    \draw (0,1.2) -- (-2.7,-.4) node[vertex]{};

    \draw (-.9,-.4) -- (-1.5,-2) node[root]{};
    \draw (-.9,-.4) -- (-.9,-2) node[root]{};
    \draw (-.9,-.4) -- (-.3,-2) node[root]{};
    \draw (0,1.2) -- (-.9,-.4) node[vertex]{};

    \node at (1.8,-.4) {\footnotesize $v_2$};
    \node at (3.6,-.4) {\footnotesize $v_1$};
    
    \draw (.9,-.4) -- (.3,-2) node[root]{};
    \draw (.9,-.4) -- (.9,-2) node[root]{};
    \draw (.9,-.4) -- (1.5,-2) node[root]{};
    \draw (0,1.2) -- (.9,-.4) node[vertex]{};

    \draw (2.7,-.4) -- (2.1,-2) node[root]{};
    \draw (2.7,-.4) -- (2.7,-2) node[root]{};
    \draw (2.7,-.4) -- (3.3,-2) node[root]{};
    \draw (0,1.2) node[vertex]{} -- (2.7,-.4) node[vertex]{};

    \draw (0,1.2) node[vertex]{} -- (4.5,-.4) node[root]{};
  \end{tikzpicture}\quad%
  \begin{tikzpicture}[thick, scale=0.5, baseline=(v.base)]
    \coordinate (v) at (0,0);

    \fill[gray!60] (0,1.2) circle (0.5);
    \fill[gray!60] (.9,-.4) circle (0.5);
    \fill[gray!60] (2.7,-.4) circle (0.5);
    \fill[gray!60] (.3,-2) circle (0.5);
    \fill[gray!60] (1.5,-2) circle (0.5);
    \fill[gray!60] (4.5,-.4) circle (0.5);

    \fill[gray!60] (.3,-2.5)--(1.5,-2.5)--(1.5,-1.5)--(.3,-1.5)--cycle;
    \node at (.9,-2.9) {\footnotesize $S_2$};
    \node at (4.5,-1.25) {\footnotesize $S'$};

    \draw (-4.5,-.4) -- (-5.1,-2) node[root]{};
    \draw (-4.5,-.4) -- (-4.5,-2) node[root]{};
    \draw (-4.5,-.4) -- (-3.9,-2) node[root]{};
    \draw (0,1.2) -- (-4.5,-.4) node[vertex]{};

    \draw (-2.7,-.4) -- (-3.3,-2) node[root]{};
    \draw (-2.7,-.4) -- (-2.7,-2) node[root]{};
    \draw (-2.7,-.4) -- (-2.1,-2) node[root]{};
    \draw (0,1.2) -- (-2.7,-.4) node[vertex]{};

    \draw (-.9,-.4) -- (-1.5,-2) node[root]{};
    \draw (-.9,-.4) -- (-.9,-2) node[root]{};
    \draw (-.9,-.4) -- (-.3,-2) node[root]{};
    \draw (0,1.2) -- (-.9,-.4) node[vertex]{};

    \draw (.9,-.4) -- (.3,-2) node[root]{};
    \draw (.9,-.4) -- (.9,-2) node[root]{};
    \draw (.9,-.4) -- (1.5,-2) node[root]{};
    \draw (0,1.2) -- (.9,-.4) node[vertex]{};

    \node at (1.8,-.4) {\footnotesize $v_2$};
    \node at (3.6,-.4) {\footnotesize $v_1$};

    \node at (0,2) {\footnotesize $v_0$};
    \draw (0,1.2) -- (4.5,-.4) node[root]{};
    \draw (0,1.2) node[vertex]{} -- (2.7,-.4) node[root]{};
  \end{tikzpicture}\quad%
  \begin{tikzpicture}[thick, scale=0.5, baseline=(v.base)]
    \coordinate (v) at (0,0);
    \fill[gray!60] (0,1.2) circle (0.5);
    \fill[gray!60] (.9,-.4) circle (0.5);
    \fill[gray!60] (2.7,-.4) circle (0.5);
    \fill[gray!60] (4.5,-.4) circle (0.5);

    \node at (4.5,-1.25) {\footnotesize $S'$};

    \node at (0,2) {\footnotesize $v_0$};
    \draw (0,1.2) -- (.9,-.4) node[root]{};

    \draw (-4.5,-.4) -- (-5.1,-2) node[root]{};
    \draw (-4.5,-.4) -- (-4.5,-2) node[root]{};
    \draw (-4.5,-.4) -- (-3.9,-2) node[root]{};
    \draw (0,1.2) -- (-4.5,-.4) node[vertex]{};

    \draw (-2.7,-.4) -- (-3.3,-2) node[root]{};
    \draw (-2.7,-.4) -- (-2.7,-2) node[root]{};
    \draw (-2.7,-.4) -- (-2.1,-2) node[root]{};
    \draw (0,1.2) -- (-2.7,-.4) node[vertex]{};

    \draw (-.9,-.4) -- (-1.5,-2) node[root]{};
    \draw (-.9,-.4) -- (-.9,-2) node[root]{};
    \draw (-.9,-.4) -- (-.3,-2) node[root]{};
    \draw (0,1.2) -- (-.9,-.4) node[vertex]{};

    \node at (1.8,-.4) {\footnotesize $v_2$};
    \node at (3.6,-.4) {\footnotesize $v_1$};

    \draw (0,1.2) -- (4.5,-.4) node[root]{};
    \draw (0,1.2) node[vertex]{} -- (2.7,-.4) node[root]{};
  \end{tikzpicture}\quad%
  \caption{$F_0$, $F_1$ and $F_2$.} \label{fig:3-trees}
\end{figure}

\subsection{Proof of \texorpdfstring{\cref{thm:neg}\ref{item:neg-2}}{Theorem 19(b)}}

We inductively deduce the negligibility of $F_k$ by that of $F_1, \dots, F_{k-1}$, where $F_k = T_{s,t-k,s'+k}$. In each inductive step, we also need to set aside the embeddings from $F_k$ to $G$ that extend the ample embeddings from $K_{1,s+1}$ to $G$ which were already dealt with in \cref{lem:ample-star}. Recall $\Ext_C(F_1, F_2, G)$ from \cref{def:extension}, and that $\ext_C(F_1, F_2, G)$ denotes its cardinality.

In the rest of the section, $s, t, p$ are fixed parameters and $n$ is a parameter that goes off to infinity. For two quantities $a, b$ with $b > 0$ that possibly depend on $n$, we write $a \lesssim b$ if there exist $C = C(s,t,p) > 0$ and $n_0 \in \N$ such that $a \le C b$ for all $n \ge n_0$.

\begin{lemma} \label{lem:ample-tree}
  Fix $s, t, p, k \in \N^+$ and $s' \in \N$ such that $s' < s$, $k \le s$ and $k < t$. Set $F_i = T_{s,t-i,s'+i}$ and $C_i = pv(F_0)^i$, for $0 \le i \le k$, and set $F_k^- = T_{s,t-k,s'}$, $\alpha = 1 - 1/\rho_{F_0}$, and $C_* = v(T_{s,t,s+1}^p)$. When $k = 1$, assume that $\alpha \ge 1 - 1/s$; and when $k \ge 2$, assume that
  \begin{equation} \label{eqn:k-cond}
    t \ge \left(1 - \tfrac{s'}{s+1}\right)k\left(k-\tfrac{1}{s}\right)(s + 2 - k) + \tfrac{1}{s}.
  \end{equation}
  For every $c > 1$ and $n$-vertex graph $G$, if every vertex in $G$ has degree between $d$ and $Kd$, where $d = cn^\alpha$ and $K = 5^{4/\alpha}$, and moreover $\amp_{C_0}(F_0, G) = 0$, then
  $$
    \amp_{C_k}(F_k, G) - \ext_{C_*}(K_{1,s+1},F_k,G)
    \lesssim \tfrac{1}{c}\inj(F_k^-, G)d^{k} + \tfrac{1}{c}nd^{e(F_k)} + \sum_{i=1}^{k-1}\amp_{C_i}(F_i, G)d^{s(i-k)}.
  $$
\end{lemma}

\begin{proof}
  As we mostly deal with embeddings to $G$, we omit $G$ in $\Inj(\cdot, G)$, $\Amp(\cdot, G)$, $\Ext(\cdot, \cdot, G)$ and their relativized versions.

  Let $v_0, v_1, \dots, v_k$ be defined for $F_0, \dots, F_k$ as in \cref{fig:3-trees}, and let $S_i$ be the set of roots which are adjacent to $v_i$ for $i \in [k]$. We view $F_i$ as a subtree of $F_{i-1}$ induced on $V(F_{i-1})\setminus S_i$. Let $\sigma$ denote an arbitrary injection from $R(F_k) \setminus \set{v_1, \dots, v_k}$ to $V(G)$, and set
  $$
    \wtas = \Amp_{C_k}(F_k; \sigma) \quad\text{and}\quad \wist = \Ext_{C_*}(K_{1,s+1}, F_k) \cap \Inj(F_k; \sigma).
  $$

  For short, denote $\vec v := (v_1, \dots, v_k)$ and $\eta(\vec v) := (\eta(v_1), \dots, \eta(v_k))$ for every $\eta \in \Inj(F_k)$. Let $\wths$ be the bipartite graph with two parts
  $$
    \wtus = \dset{\eta(v_0)}{\eta \in \wtas} \quad \text{and} \quad \wtws = \dset{\eta(\vec v)}{\eta \in \wtas}
  $$
  whose edge set is given by
  $$
    \wths = \dset{(\eta(v_0), \eta(\vec v))}{\eta \in \wtas}.
  $$

  \setcounter{claim}{0}
  \begin{claim} \label{claim:new-1}
    The size of $\wtas$ is bounded by that of $\wths$ as follows:
    $$
      \abs{\wtas} - \abs{\widetilde{I}^\times_\sigma} \lesssim \abs{\wths}.
    $$
  \end{claim}
  \begin{claimproof}[Proof of \cref{claim:new-1}]
    In view of the definition of $\widetilde{I}^\times_\sigma$, $\wtas \setminus \widetilde{I}^\times_\sigma$ contains no extension of any $C_*$-ample embedding from $K_{1,s+1}$ to $G$. Therefore for every edge $(u, \vec w)$ in $\wths$, there are at most $C_*^{t-k}$ many $\eta \in \wtas \setminus \widetilde{I}^\times_\sigma$ with $(\eta(v_0), \eta(\vec v)) = (u,\vec w)$.
  \end{claimproof}

  Sample a subset $\us$ of $\wtus$ of size $m_0$ chosen uniformly at random, where $m_0$ will be chosen later. We denote $\hs$ the bipartite subgraph $\hs$ of $\wths$ induced on $\us \cup \wtws$, and we partition $\hs$ into $\hs^-$ and $\hs^+$, where $\hs^-$ consists of edges $(u, \vec w)$ in $\hs$ such that $\vec w$ has degree at most $sk$ in $\hs$, and $\hs^+$ is the complement of $\hs^-$ in $\hs$. We estimate the number of edges in $\hs^-$ and $\hs^+$ in the following two claims respectively.

  \begin{claim} \label{claim:new-2}
    For every $\sigma$, the number of edges in $\hs^-$ satisfies
    $$
      \abs{\hs^-}d^{sk} \lesssim n^{sk} + \sum_{i=1}^{k-1}\abs*{\dset{\eta \in \Amp_{C_i}(F_i; \sigma)}{\eta(v_0) \in U_\sigma}}d^{si}.
    $$
  \end{claim}
  \begin{claimproof}[Proof of \cref{claim:new-2}]
    We define a subst $B_\sigma$ of $\Inj(F_0; \sigma)$ as follows. For every edge $(u, \vec w)$ in $\hs^-$, we choose some $\eta \in \wtas$ with $(u, \vec w) = (\eta(v_0), \eta(\vec v))$, and then this chosen $\eta$ gives rise to $(1-o(1))d^{sk}$ many $\eta' \in \Inj(F_0; \sigma)$ such that $\eta' \supseteq \eta$ and $(u, \vec w) = (\eta'(v_0), \eta'(\vec v))$. Finally, we collect these $\eta'$ in $B_\sigma$.
    
    Note that
    \begin{equation} \label{eqn:dsk}
      (1-o(1))\abs{\hs^-}d^{sk} \le \abs{B_\sigma},
    \end{equation}
    and $B_\sigma$ has the distinctness property in the sense that
    \begin{equation} \label{eqn:distinctness}
      \text{no two distinct embeddings in $B_\sigma$ are identical on }\set{v_0, v_1, \dots, v_k} \cup S_1 \cup \dots \cup S_k.
    \end{equation}
      
    Let $\sigma'$ denote an arbitrary injection from $R(F_0)$ to $V(G)$ such that $\sigma' \supseteq \sigma$, and define $B_{\sigma'} = B_\sigma \cap \Inj(F_0; \sigma')$. We claim that, for every $I \subsetneq [k]$, the cardinality of
    \begin{multline*}
      B_{\sigma'}^I := \dset{\eta' \in B_{\sigma'}}{\text{there exist }\eta' = \eta'_1, \eta'_2, \dots, \eta'_{C_i} \in B_{\sigma'} \text{ such that } \\ \eta_1'(\vec v), \dots, \eta_{C_i}'(\vec v) \text{ form a sequential sunflower of size }C_i \text{ whose kernel is }I}.
    \end{multline*} satisfies
    \begin{equation} \label{eqn:isigma}
      \sum_{\sigma'}\abs{B_{\sigma'}^I} \le \abs{\dset{\eta \in \Amp_{C_i}(F_i; \sigma)}{\eta(v_0) \in U_\sigma}}(Kd)^{si}, \quad\text{where }i = \abs{I}.
    \end{equation}
    
    Without loss of generality, we may assume that $I = [k] \setminus [k-i]$ for some $i \in \set{0, \dots, k-1}$. Clearly $\sum_{\sigma'}\abs{B_{\sigma'}^I} = \abs{\bigcup_{\sigma'}B_{\sigma'}^I}$. For every $\sigma'$ and $\eta' \in B_{\sigma'}^I \subseteq \Inj(F_0;\sigma)$, we claim that the restriction, say $\eta \in \Inj(F_i;\sigma)$, of $\eta'$ to $V(F_i)$ is in $\Amp_{C_i}(F_i;\sigma)$. Assuming the claim, as there are at most $(Kd)^{si}$ ways to extend $\eta \in \Amp_{C_i}(F_i;\sigma)$ to an embedding from $F_0$ to $G$, \cref{eqn:isigma} follows immediately.
    
    To see that $\eta$ is $C_i$-ample, by the definition of $B_{\sigma'}^I$, there exist $\eta' = \eta'_1, \eta'_2, \dots, \eta'_{C_i} \in B_{\sigma'}$ such that $\eta_1'(\vec v), \dots, \eta_{C_i}'(\vec v)$ form a sequential sunflower of size $C_i$ with kernel $I = [k] \setminus [k-i]$. Unwinding the definition of a sequential sunflower, for two distinct $j_1, j_2 \in [C_i]$, we know that $\eta'_{j_1}$ and $\eta'_{j_2}$ are identical on $\set{v_{k-i+1}, \dots, v_{k}}$, but $\set{\eta'_{j_1}(v_1), \dots, \eta'_{j_1}(v_{k-i})}$ and $\set{\eta'_{j_2}(v_1), \dots, \eta'_{j_2}(v_{k-i})}$ are disjoint. For every $j \in [C_2]$, since $\eta_j' \in B_{\sigma'} \subseteq B_{\sigma}$, we know according to our choice of $B_\sigma$ that the restriction of $\eta_j$ to $V(F_k)$ is a $C_k$-ample embedding from $F_k$ to $G$. Thus, using the assumption that $C_k \ge C_iv(F_i)$, we can greedily modify $\eta_1', \dots, \eta_{C_i}'$ one at a time so that the images of $V(F_k) \setminus R(F_k)$ under $\eta_1', \dots, \eta_{C_i}'$ are pairwise disjoint. One can now easily check that the restrictions, say $\eta = \eta_1, \dots, \eta_{C_i}$, of $\eta'_1, \dots, \eta_{C_i}'$ to $V(F_i)$ satisfy that they are identical on $R(F_i)$, and the images of $V(F_i)\setminus R(F_i)$ under $\eta_1, \dots, \eta_{C_i}$ are pairwise disjoint.

    Finally we estimate the cardinality of
    $$  
      B_{\sigma'}^\times := B_{\sigma'} \setminus \bigcup_{I \subsetneq [k]} B_{\sigma'}^I.
    $$
    Set $W = \dset{\eta'(\vec v)}{\eta' \in B_{\sigma'}^\times}$.
    For every sequence $\vec w \in W$, as the degree of $\vec w$ is at most $sk$ in $\hs$, together with the distinctness property \cref{eqn:distinctness} of $B_\sigma \supseteq B_{\sigma'}$, we know that $\abs{B_{\sigma'}^\times} \le sk\abs{W}$. By the definitions of $B_{\sigma'}^\times$ and $B_{\sigma'}^I$, one can check that $W$ contains no sequential sunflower of size $\max(C_0, \dots, C_{k-1}) = pv(F_0)^{k-1}$. Thus \cref{lem:sunflower} implies $\abs{B_{\sigma'}^\times} \le sk\abs{W} \le K_0$ for some positive constant $K_0 = K_0(s,t,p)$, and so
    $$
      \abs{B_{\sigma'}} \le K_0 + \sum_{I\subsetneq [k]}\abs{B_{\sigma'}^I}.
    $$
    Because the total number of $\sigma'\colon R(F_0) \to V(G)$ such that $\sigma' \supseteq \sigma$ is at most $n^{sk}$, summing the last inequality over all $\sigma'$, together with \cref{eqn:dsk}, yields
    $$
      (1-o(1))\abs{\hs^-}d^{sk} \le \abs{B_\sigma} \le K_0n^{sk} + \sum_{\sigma'}\sum_{I\subsetneq [k]}\abs{B_{\sigma'}^I},
    $$
    which implies the desired inequality in view of \cref{eqn:isigma} and the assumption that $\amp_{C_0}(F_0) = 0$.
  \end{claimproof}

  \begin{claim} \label{claim:new-3}
    For every $\sigma$, the number of edges in $\hs^+$ satisfies
    $$
      \abs{\hs^+} \lesssim \begin{cases}
        m_0n^{1-1/s} & \text{if } k = 1; \\
        m_0^{(sk-1)\left(1-\frac{k-1}{s+1}\right)+1}d^{k-1} + m_0^{sk+\frac{s(k-1)}{s+1}} & \text{otherwise}.
      \end{cases}
    $$
  \end{claim}

  \begin{claimproof} [Proof of \cref{claim:new-3}]
    Let $U_0$ denote an arbitrary vertex subset of $U_\sigma$ of size $sk$ in $\hs$, and denote $N_{\hs^+}(U_0) \subseteq \wtws$ the common neighborhood of $U_0$ in $\hs^+$. Let $W(U_0)$ be the $k$-uniform hypergraph defined by
    $$
      W(U_0) = \dset{\set{w_1, \dots, w_k}}{(w_1, \dots, w_k) \in N_{\hs^+}(U_0)}.
    $$
    
    Observe that $W(U_0)$ contains no matching of size $C_0 = p$. Indeed, suppose on the contrary that $W(U_0)$ contains a matching $e_1, \dots, e_p$ of size $p$. We can find a $p$-ample embedding from $F_0$ to $G$ as follows, which would contradict with the assumption $\amp_{C_0}(F_0) = 0$. Since $e_1 \in W(U_0)$, we know that $e_1$ is the image of $\set{v_1, \dots, v_k}$ under some embedding in $\Amp_{C_k}(F_k; \sigma)$. Because $C_k > sk$, we can find some $\eta_1 \in \Amp_{C_k}(F_k; \sigma)$ with $\set{\eta_1(v_1), \dots, \eta_1(v_k)} = e_1$ such that $\eta_1(V(F_k)\setminus R(F_k))$ does not intersect $U_0$. We then extend $\eta_1$ to $\eta_1' \in \Inj(F_0)$ by mapping $S_1 \cup \dots \cup S_k$ to $U_0$ additionally.
    
    To see that $\eta_1'$ is in fact $p$-ample, we greedily build a sequence of embeddings $\eta_1', \dots, \eta_p'$ in $\Inj(F_0)$ such that they are identical on $R(F_0)$, and the images of $V(F_0)\setminus R(F_0)$ under $\eta_1', \dots, \eta_p'$ are pairwise disjoint. Suppose we have built $\eta_1', \dots, \eta_j'$ for some $j < p$. Similar to how we found $\eta_1'$, because $C_k > sk + jv(F_0)$, we can find some $\eta_{j+1} \in \Amp_{C_k}(F_k; \sigma)$ with $\set{\eta_{j+1}(v_1), \dots, \eta_{j+1}(v_k)} = e_{j+1}$ such that $\eta_1(V(F_k)\setminus R(F_k))$ does not intersect $U_0 \cup \bigcup_{i \le j}\eta'_i(V(F_0))$. We extend $\eta_j$ to $\eta_j'$ by mapping $S_1 \cup \dots \cup S_k$ to $U_0$ the same way as $\eta_1'$.
    
    Now we treat the $k = 1$ case and the $k \ge 2$ case separately.

    \paragraph{Case 1: $k = 1$.}
    In this case, $W(U_0)$ is a $1$-uniform hypergraph, and it contains less than $p$ vertices for every $U_0$. Therefore $\hs^+$ contains no complete bipartite subgraph with $s$ vertices in $\us$ and $p$ vertices in $\wtws$. \cref{lem:kst} shows that $\abs{\hs^+} \lesssim \abs{\us}\abs{\wtws}^{1-1/s}$, which implies the desired inequality in view of the fact that $\abs{\wtws} \le n$.

    \paragraph{Case 2: $k \ge 2$.}
    Using the assumption that $d_{\hs^+}(\vec w) > sk$ for every $\vec w \in \wtws$, a simple double counting argument shows that
    $$
      \abs{\hs^+} = \sum_{\vec w \in \wtws} d_{\hs^+}(\vec w) \le \sum_{\vec w \in \wtws} \binom{d_{\hs^+}(\vec w)}{sk} = \sum_{U_0}\abs{N_{\hs^+}(U_0)},
    $$
    which, together with the fact that $\abs{N_{\hs^+}(U_0)} \le k!\abs{W(U_0)}$, implies that
    $$
      \abs{\hs^+} \lesssim \sum_{U_0}\abs{W(U_0)}.
    $$

    For convenience, denote $N(U_0)$ the vertex set of the $k$-uniform hypergraph $W(U_0)$. As $W(U_0)$ contains no matching of size $p$, clearly we have $\abs{W(U_0)} \le k(p - 1)\abs{N(U_0)}^{k-1}$, and so
    $$
      \abs{\hs^+} \lesssim \sum_{U_0}\abs{N(U_0)}^{k-1}.
    $$

    It suffices to estimate $\sum_{U_0}\abs{N(U_0)}^{k-1}$. Clearly $\abs{N(U_0)} \le Kd$, and so $\sum_{U_0} \abs{N(U_0)}^{k-1} \le m_0^{sk}(Kd)^{k-1}$, which gives the following simple bound on $\abs{\hs^+}$,
    \begin{equation} \label{eqn:hs-weak}
      \abs{\hs^+} \lesssim m_0^{sk}d^{k-1}.
    \end{equation}  
    To get a better estimate on $\sum_{U_0}\abs{N(U_0)}^{k-1}$, we squeeze a bit more out of the assumption that $G$ contains no $F_0^p$ as a subgraph by iteratively applying \cref{lem:drc}.\footnote{Had we used the simple bound \cref{eqn:hs-weak} on $\abs{\hs^+}$, we could have removed the rest proof of \cref{claim:new-3} together with \cref{lem:drc}. The tradeoff that comes with this simplification is a condition on $t$ that is more restricted than \cref{eqn:k-cond}.}
    
    Let $V(\wtws) \subseteq V(G)$ be the set of vertices that ever appear in any sequence in $\wtws$. For every subset of $U \subseteq U_\sigma$, we denote $N'_G(U)$ the set of vertices in $V(\wtws)$ that are adjacent to every vertex in $U$ in the graph $G$. We prove inductively for $1 \le i \le sk$ that
    \begin{equation} \label{eqn:nhbr}
      \sum_{U \subseteq U_\sigma \colon \abs{U} = i}\abs{N_G'(U)}^{k-1} \lesssim m_0^{(i-1)\left(1-\frac{k-1}{s+1}\right)+1} d^{k-1} + m_0^{i+\frac{s(k-1)}{s+1}}.
    \end{equation}
    Notice that $N(U_0) \subseteq N'_G(U_0)$ for every $U_0$. In particular, taking $i = sk$ in \cref{eqn:nhbr} gives
    $$
      \sum_{U_0} \abs{N(U_0)}^{k-1} \lesssim m_0^{(sk-1)\left(1-\frac{k-1}{s+1}\right)+1}d^{k-1} + m_0^{sk+\frac{s(k-1)}{s+1}},
    $$
    which implies the desired inequality in \cref{claim:new-3}.
    
    The base case $i = 1$ is evident as the maximum degree of $G$ is at most $Kd$. For the inductive step, consider an arbitrary $U \subseteq U_\sigma$ of size $i - 1$ and denote $u$ an arbitrary vertex in $U_\sigma \setminus U$. Clearly $\abs{N_G'(U \cup \set{u})} = d_{G(U)}(u)$, where $G(U)$ is the bipartite subgraph of $G$ induced on $U_\sigma$ and $N_G'(U)$. Observe that there is no embedding $\eta \in \Inj(T_{s,t,0}, G(U))$ such that $\eta(R(T_{s,t,0}^p)) \subseteq U_\sigma$, because otherwise one can extend $\eta \in \Inj(T_{s,t,0}, G(U))$ to $\eta' \in \Inj(F_0^p)$ such that $\eta'$ and $\sigma$ are identical on $S'(F_0^p) = S'(F_0)$ (see \cref{fig:vertex-subsets} for the definitions of $S'(F_0)$ and $Q(F_0)$). As every vertex in $Q(T_{s,t,0}^p)$ has degree $s+1$, and $\abs{U_\sigma} = m_0$, \cref{lem:drc} shows that
    $$
      \sum_{u \in U_\sigma\setminus U}\abs{N_G'(U \cup \set{u})}^{k-1} = \sum_{u \in U_\sigma \setminus U}d_{G(U)}(u)^{k-1} \lesssim \left(m_0^{k-1} + \abs{N_G'(U)}^{k-1}\right)m_0^{1-\frac{k-1}{s+1}}.
    $$
    Let $U'$ denote an arbitrary subset of $U_\sigma$ of size $i$.
    Summing the above inequality over all $U \subseteq U_\sigma$ of size $i-1$, we obtain from the inductive hypothesis that
    \begin{align*}
      \sum_{U'}\abs{N_G'(U')}^{k-1} & \lesssim \sum_{U} \left( m_0^{1+\frac{s(k-1)}{s+1}} + m_0^{1-\frac{k-1}{s+1}}\abs{N_G'(U)}^{k-1}\right) \\
      & \lesssim m_0^{i+\frac{s(k-1)}{s+1}} + m_0^{1-\frac{k-1}{s+1}}\sum_{U}\abs{N_G'(U)}^{k-1} \\
      & \lesssim m_0^{i+\frac{s(k-1)}{s+1}} + m_0^{(i-1)\left(1-\frac{k-1}{s+1}\right)+1}d^{k-1} + m_0^{i + \frac{(s-1)(k-1)}{s+1}},
    \end{align*}
    which becomes \cref{eqn:nhbr} after noticing that $m_0^{i+\frac{s(k-1)}{s+1}}$ subsumes $m_0^{i + \frac{(s-1)(k-1)}{s+1}}$.
  \end{claimproof}

  Before we assemble \cref{claim:new-1,claim:new-2,claim:new-3} together, we observe that
  \begin{equation} \label{eqn:wtus}
    \abs{\wtus} \le \inj(F_k^-;\sigma).
  \end{equation}
  Indeed, since every $u \in \wtus$ corresponds to $\eta \in \Inj(F_k^-; \sigma)$ such that $\eta(v_0) = u$, where $F_k^-$ is the subgraph of $F_k$ induced on $V(F_k) \setminus \set{v_1, \dots, v_k}$, clearly we have \cref{eqn:wtus}.
  Like in the proof of \cref{claim:new-3}, we treat the $k = 1$ case and the $k \ge 2$ case separately.
  
  \paragraph{Case 1: $k = 1$.} We simply take $m_0 = \abs{\wtus}$, in other words, $U_\sigma = \wtus$. Notice that every vertex $\vec w \in \wtws$ has degree at least $C_k$ in $\wths$ because $\vec w = \eta(\vec v)$ for some $C_k$-ample $\eta$ from $F_k$ to $G$. Therefore $\wths = \hs^+$. By \cref{claim:new-1,claim:new-3} and the assumption $1 - 1/s \le \alpha$, we obtain for every $\sigma$ that
  $$
    \abs{\wtas} - \abs{\wist} \lesssim \abs{\wths} = \abs{\hs^+} \lesssim \abs{\wtus}n^{1-1/s}.
  $$
  Using the assumption $1-1/s \le \alpha$, we estimate that $n^{1-1/s} \le n^\alpha = \tfrac{1}{c}d^k$. Therefore, in view of \cref{eqn:wtus}, we have $\abs{\wtas} - \abs{\wist} \lesssim \tfrac{1}{c}\inj(F_k^-; \sigma)d^k$. Summing over all injections $\sigma\colon R(F_k) \setminus \set{v_1, \dots, v_k} \to V(G)$ yields
  $$
    \amp_{C_k}(F_k) - \ext_{C_*}(K_{1,s+1},F_k) \lesssim \tfrac{1}{c} \inj(F_k^-)d^k.
  $$
  
  \paragraph{Case 2: $k \ge 2$.}
  We take $m_0 = \floor{n^{(s - (s+1)\alpha)k}}$. As $s' < s$, one can check that $\rho_{F_0} < s+1$, and so $\alpha < s/(s+1)$, which implies $s - (s+1)\alpha > 0$. Hence $m_0 = \omega(1)$. We claim that the condition \cref{eqn:k-cond} on $t$ implies that
  \begin{equation} \label{eqn:m0-n}
    m_0^{(sk-1)\left(1-\frac{k-1}{s+1}\right)} \le n^\alpha \quad\text{and}\quad m_0^{sk+\frac{s(k-1)}{s+1}} \le m_0n^\alpha d^{k-1}.
  \end{equation}
  Indeed, using $\alpha = 1 - 1/\rho_{F_0} = (st+s'-1)/(st+t+s')$, one can check that \cref{eqn:k-cond} is equivalent to
  $$
    (s-(s+1)\alpha)k(sk-1)\left(1-\tfrac{k-1}{s+1}\right) \le \alpha,
  $$
  which implies the first inequality in \cref{eqn:m0-n}. To see that the second inequality follows from the first inequality in \cref{eqn:m0-n}, in view of the fact that $n^\alpha d^{k-1} \ge n^{k\alpha}$ as $d = cn^\alpha$ and $c > 1$, it suffices to check that $(sk-1)(1-(k-1)/(s+1)) \ge (sk+s(k-1)/(s+1)-1)/k$, which is equivalent to $sk(k - 1)(s-k + 1) + (k-1)^2 \ge 0$, which clearly holds.
  
  Using \cref{eqn:m0-n}, we can simplify \cref{claim:new-3} to $\abs{\hs^+} \lesssim m_0n^\alpha d^{k-1}$. Combining this with \cref{claim:new-2}, we obtain for every $\sigma$ that
  $$
    \abs{\hs} = \abs{\hs^-} + \abs{\hs^+}
    \lesssim n^{sk}d^{-sk} + m_0n^\alpha d^{k-1} + \sum_{i=1}^{k-1} \abs{\dset{\eta \in \Amp_{C_i}(F_i; \sigma)}{\eta(v_0) \in U_\sigma}}d^{s(i-k)}.
  $$
  Recall that $U_\sigma$ is a subset of $\wtus$ of size $m_0$ chosen uniformly at random, and $\hs$ is the bipartite subgraph of $\wths$ induced on $U_\sigma \cup \wtws$. Observe that the expectation of $\abs{\hs}$ is $m_0\abs{\wths}/\abs{\wtus}$, and the expectation of $\abs{\dset{\eta \in \Amp_{C_i}(F_i; \sigma)}{\eta(v_0) \in U_\sigma}}$ is $m_0\amp_{C_i}(F_i;\sigma)/\abs{\wtus}$. Thus taking the expectation of the above inequality, and then multiplying both sides by $\abs{\wtus}/m_0$, gives
  \begin{equation} \label{eqn:ineq-new-1}
    \abs{\wths} \lesssim \abs{\wtus}\left(m_0^{-1}n^{sk}d^{-sk} + n^\alpha d^{k-1}\right) + \sum_{i=1}^{k-1} \amp_{C_i}(F_i;\sigma)d^{s(i-k)}.    
  \end{equation}
  In case $\abs{\wtus} < m_0$, we could not take $\us \subseteq \wtus$ of size $m_0$, and we instead simply bound
  \begin{equation} \label{eqn:ineq-new-2}
    \abs{\wths} \le m_0(Kd)^k \lesssim m_0d^k.
  \end{equation}

  Using $d = cn^\alpha$, one can routinely check that $m_0^{-1}n^{sk}d^{-sk}$ and $n^\alpha d^{k-1}$ in \cref{eqn:ineq-new-1} are at most $(1+o(1))d^k/c$, and $m_0d^k$ in \cref{eqn:ineq-new-2} is at most $n^{1-(\abs{R(F_k)-k})}d^{e(F_k)}/c$. Therefore \cref{claim:new-1}, \cref{eqn:wtus,eqn:ineq-new-1,eqn:ineq-new-2} imply, regardless of whether $\abs{\wtus} \ge m_0$, that
  $$
    \abs{\wtas} - \abs{\wist} \lesssim \abs{\wths} \lesssim \tfrac{1}{c}\inj(F_k^-;\sigma)d^{k} + \tfrac{1}{c}n^{1-(\abs{R(F_k)}-k)}d^{e(F_k)} + \sum_{i=1}^{k-1}\amp_{C_i}(F_i;\sigma)d^{s(i-k)}.
  $$
  Summing over all injections $\sigma\colon R(F_k) \setminus \set{v_1, \dots, v_k} \to V(G)$ yields that
  $$
    \amp_{C_k}(F_k) - \ext_{C_*}(K_{1,s+1}, F_k) \lesssim
    \tfrac{1}{c}\inj(F_k^-)d^{k} + \tfrac{1}{c}nd^{e(F_k)} + \sum_{i=1}^{k-1}\amp_{C_i}(F_i)d^{s(i-k)}.
  $$
  This completes the proof of \cref{lem:ample-tree}.
\end{proof}

\begin{proof}[Proof of \cref{thm:neg}\ref{item:neg-2}]
  Assume that $s - s' \ge 1$. Denote $F_k = T_{s,t-k,s'+k}$ for $0 \le k \le s - s'$. In particular, $F_0 = T_{s,t,s'}$. Let $p \in \N^+$ and $C_* = v(T_{s,t,s+1}^p)$. Set $C_k = p v(F_0)^i$ for $k \le s-s'$. Let $c_0 \ge 1 / \eps$ be the constant to be chosen. Suppose that $c > c_0$ and $G$ is an $n$-vertex graph with $n \ge n_0(c)$ such that every vertex in $G$ has degree between $d$ and $Kd$, where $d = cn^\alpha$ and $K = 5^{4/\alpha}$, and moreover $\amp_{C_0}(F_0, G) = 0$. We break the rest of the proof into two cases.

  \paragraph{Case 1:} $k = 1$. Let $c_0$ be at least the constant already obtained from \cref{thm:neg}\ref{item:neg-1}. By the choice of $c_0$, we know that $\amp_{C_*}(K_{1,s+1}, G) \le \eps nd^{e(K_{1,s+1})}$. Since $F_0$ is balanced, by \cref{lem:balance-condition}, $\rho_{F_0} \ge s$, which implies that $1 - 1/s \le \alpha$, where $\alpha = 1 - 1/\rho_{F_0}$. By \cref{lem:ample-tree} and the assumption $1/c < \eps$, we obtain
  $$
    \amp_{C_k}(F_k, G) \lesssim \eps \inj(F_k^-, G)d^k + \eps nd^{e(F_k)} + \ext_{C_*}(K_{1,s+1}, F_k, G).
  $$
  Since $\inj(F_k^-, G) \le n(Kd)^{e(F_k^-)} \lesssim nd^{e(F_k)-k}$, and moreover
  $$
    \ext_{C_*}(K_{1,s+1}, F_k, G) \le \inj(K_{1,s+1}, F_k)\amp_{C_*}(K_{1,s+1}, G)(Kd)^{e(F_k) - e(K_{1,s+1})} \lesssim \eps nd^{e(F_k)},
  $$
  we estimate $\amp_{C_k}(F_k, G) \lesssim \eps nd^{e(F_k)}$, which implies the desired inequality in \cref{thm:neg}\ref{item:neg-2}.

  \paragraph{Case 2:} $2 \le k \le s - s'$. By induction, there exists $c_0 \ge 1/\eps$ such that $\amp_{C_i}(F_i, G) \le \eps nd^{e(F_i)}$ for every $1 \le i < k$. Note that the assumption \cref{eqn:t-cond} in \cref{thm:neg} ensures the condition \cref{eqn:k-cond} in \cref{lem:ample-tree}. By \cref{lem:ample-tree} and the assumption $1/c < \eps$, we similarly obtain that $\amp_{C_k}(F_k, G) \lesssim \eps nd^{e(F_k)}$.
\end{proof}

\section*{Acknowledgements}
We are grateful to Boris Bukh and Ryan Alweiss for comments on the earlier versions of this paper. We are thankful to the anonymous referees for their astute and useful comments.

\bibliographystyle{plain}
\bibliography{turan_exponent}

\end{document}